\documentclass[11pt,a4paper]{amsart}

\usepackage{preamble}

\title[\texorpdfstring{$\sl(2)$}{sl(2)}-type singular fibres]{\texorpdfstring{$\sl(2)$}{sl(2)}-type singular fibres of the symplectic and odd orthogonal Hitchin system
}
\author{Johannes Horn}
\begin{document}

\begin{abstract} We define and parametrize so-called $\sl(2)$-type fibres of the \linebreak $\Sp(2n,\C)$- and $\SO(2n+1,\C)$-Hitchin system. These are (singular) Hitchin fibres, such that spectral curve establishes a two-sheeted covering of a second Riemann surface $Y$. This identifies the $\sl(2)$-type Hitchin fibres with fibres of an $\SL(2,\C)$- respectively $\PSL(2,\C)$-Hitchin map on $Y$.

Building on results of \cite{Ho1}, we give a stratification of these singular spaces by semi-abelian spectral data, study their irreducible components and obtain a global description of the first degenerations. We will compare the semi-abelian spectral data of $\sl(2)$-type Hitchin fibres for the two Langlands dual groups. This extends the well-known Langlands duality of regular Hitchin fibres to $\sl(2)$-type Hitchin fibres. Finally, we will construct solutions to the decoupled Hitchin equation for $\sl(2)$-type fibres of the symplectic and odd orthogonal Hitchin system. We conjecture these to be limiting configurations along rays to the ends of the moduli space. 
\end{abstract}

\maketitle


\section{Introduction}\label{chap:intro}

For more than thirty years, the study of moduli spaces of Higgs bundles is a very active research area located at the crossroads of algebraic, complex and differential geometry with the theory of integrable systems and surface group representations. One major reason for the ongoing interest in these moduli spaces is their extremely rich geometry. They were introduced by Hitchin \cite{Hi87a} as examples of non-compact hyperkähler spaces. They are homeomorphic to moduli spaces of flat $G$-bundles on $X$ by the famous non-abelian Hodge correspondence \cite{Hi87a,Do87,Si88,Co88}. And most importantly for the present work, they have a dense subset carrying the structure of an algebraically completely integrable system - the so-called Hitchin system \cite{Hi87b}. 



By definition, the Higgs bundle moduli space $\M_G$ on a Riemann surface $X$ associated to a complex reductive linear group $G$ is a moduli space of pairs $(E,\Phi)$. Here $E$ is a holomorphic $G$-vector bundle on $X$ and $\Phi$ is holomorphic one-form valued in $\g$, called the Higgs field. $\M_G$ has a complex symplectic structure on its smooth locus and 
the Hitchin map 
\[ \H_G : \M_G \rightarrow B_G.
\] defines a proper, surjective, holomorphic map to a complex vector space $B_G$ of half the dimension of $\M_G$, referred to as the Hitchin base. 
Hitchin showed for the classical groups \cite{Hi87b} and Scognamillo for all complex reductive groups \cite{Sc98}, that on a dense subset $B_G^\reg \subset B_G$ the fibres of the Hitchin map are torsors over abelian varieties. Thereby, the pre-image of the regular locus $B_G^\reg$ under the Hitchin map is an algebraically completely integrable system, nowadays called the Hitchin system. 

To identify the Hitchin fibers over the regular locus with abelian varieties one introduces spectral data. The Hitchin map applied to a Higgs bundle $(E,\Phi)$ computes the eigenvalues of the Higgs field $\Phi$. These eigenvalues are decoded in the spectral curve $\Sigma$, a covering of the original Riemann surface $X$. The eigenspaces determine a line bundle on the spectral curve. For a point in the regular locus $B_G^\reg$ the spectral curve is smooth. In this case, the moduli spaces of eigen line bundles are the classical examples of abelian varieties, most importantly Jacobian and Prym varieties. 

The Hitchin fibration played a major role in two recent developments in the theory of Higgs bundle moduli spaces: Firstly, in the study of the asymptotic of the hyperkähler metric \cite{MSWW19} and secondly, in the Langlands duality of Higgs bundle moduli spaces \cite{DoPa12}. Both results were considered on the regular locus of the Hitchin map and it is an interesting question how they extend to the singular locus (see \cite{AEFS18}). In this paper, we do the first steps in this direction. 


%

\subsection*{Singular Hitchin fibres of \texorpdfstring{$\sl(2)$}{sl(2)}-type}
We introduce and study the class of $\sl(2)$-type Hitchin fibres of the $\Sp(2n,\C)$- and $\SO(2n+1,\C)$-Hitchin system. This class of (singular) Hitchin fibres is distinguished by the singularities of the spectral curve, such that for $n=2$ all fibres are of $\sl(2)$-type (see Definition \ref{def:sl(2)-type_Sp} resp. \ref{def:sl(2)-type_SO} for precise definitions). For $\SL(2,\C)$, the singular Hitchin fibres were studied in \cite{Sch98,GO13} using the Beauville-Narasimhan-Ramanan correspondence \cite{BNR89}. In \cite{Ho1}, the author developed a more direct approach introducing semi-abelian spectral data. These consist of an abelian torsor over the Prym variety of the normalised spectral curve and non-abelian coordinates parametrising local deformations of the Higgs bundle at the singularities of the spectral curve.

For Hitchin fibres of $\sl(2)$-type the spectral curve $\Sigma$ defines a two-sheeted covering over another Riemann surface $Y$. The main result of this work identifies the $\Sp(2n,\C)$- and $\SO(2n+1,\C)$-Hitchin fibres of $\sl(2)$-type with $\SL(2,\C)$- respectively $\PGL(2,\C)$-Hitchin fibres of a moduli space of twisted Higgs bundles on $Y$ (Theorem \ref{Theo:isom:Hitchinfibres} and \ref{theo:so:biholo}). This allows to extend the results of \cite{Ho1} to $\sl(2)$-type Hitchin fibres.

\begin{theo}[Theorem \ref{theo:Sp(2n):strat}, \ref{theo:so(2n+1):fibreing}]\label{theo:intro1} Let $G=\Sp(2n,\C)$ or $G=\SO(2n+1,\C)$. Let $b \in B_G$ with irreducible and reduced spectral curve of $\sl(2)$-type. Then there exists a stratification 
\[ \H_G^{-1}(b) = \bigsqcup_{i \in I} \S_i
\] by finitely many locally closed subsets $\S_i$, such that every stratum $\S_i$ is a finite-to-one covering of a $(\C^*)^{r_i} \times \C^{s_i}$-bundle over an abelian torsor $T(b)$. 
\end{theo}
When the spectral curve $\Sigma$ is smooth, it is of $\sl(2)$-type. Then the stratification is trivial and this result gives a new approach to the identification of regular fibres of the symplectic and odd orthogonal Hitchin system with abelian torsors originally obtained in \cite{Hi07}.

The abelian torsor parametrises the eigen line bundles of $(E,\Phi) \in \H^{-1}_G(b)$ and will be referred to as the abelian part of the spectral data. The $(\C^*)^{r_i} \times \C^{s_i}$-fibres, the non-abelian part of the spectral data, parametrize Hecke transformations of the Higgs bundle at the singularities of the spectral curve. 


The stratification of Theorem \ref{theo:intro1} contains a unique, open and dense stratum $\S_0 \subset \H^{-1}_G(b)$. This dense stratum is compactified by lower dimensional strata distinguished from $\S_0$ by a lower dimensional moduli space of Hecke parameters. For the unique closed stratum of lowest dimension, this parameter space is a point and hence this stratum is an abelian torsor. 

In the second part of \cite{Ho1}, it was studied how the fibres glue together to form the singular Hitchin fibre. Let us describe the first degeneration in more detail. For $G=\SL(2,\C)$, the Hitchin base is the vector space of quadratic differentials $H^0(X,K_X^2)$. In this setting, the examples we want to consider are Hitchin fibres over a quadratic differential $q \in H^0(X,K_X^2)$ with a single zero of order $2$, such that all other zeros are simple. For $\Sp(2n,\C)$ and $\SO(2n+1,\C)$, there are singular Hitchin fibres like this, for all $n \in \N$. 

In this example, we have two strata each isomorphic to a $(\C^*)^{r_i} \times \C^{s_i}$-bundle over the abelian torsor $T(q)$ with exponents given by
\[ \S_0: (r_0=1,s_0=0), \quad \S_1: (r_1=0,s_1=0).
\] In Figure \ref{fig:singfibre2}, we sketched the situation by compressing the abelian part of the spectral data to a circle. On the left hand side, we see a sketch of the open and dense stratum $\S_0$, where the $\C^*$-fibres are depicted by little tunnels. We obtain the singular Hitchin fibre by gluing the two missing points of the $\C^*$-fibre to the abelian torsor in a twisted way. Indeed, the Higgs bundles corresponding to the points zero and infinity do not have the same eigen line bundle and hence do not correspond to the same point on the abelian torsor. In particular, the fibring over the abelian torsor does not extend to $\Hit^{-1}(q)$. 
This example can be also found in \cite{GO13,HiCritical} for the $\SL(2,\C)$-case. 
More generally, we will give a global description of the first degenerations up to normalisation in examples \ref{exam:sp(2n,C)_first_deg} and \ref{exam:so(2n+1,C)_first_deg}. 
\begin{figure}[h]
\includegraphics[scale=0.21]{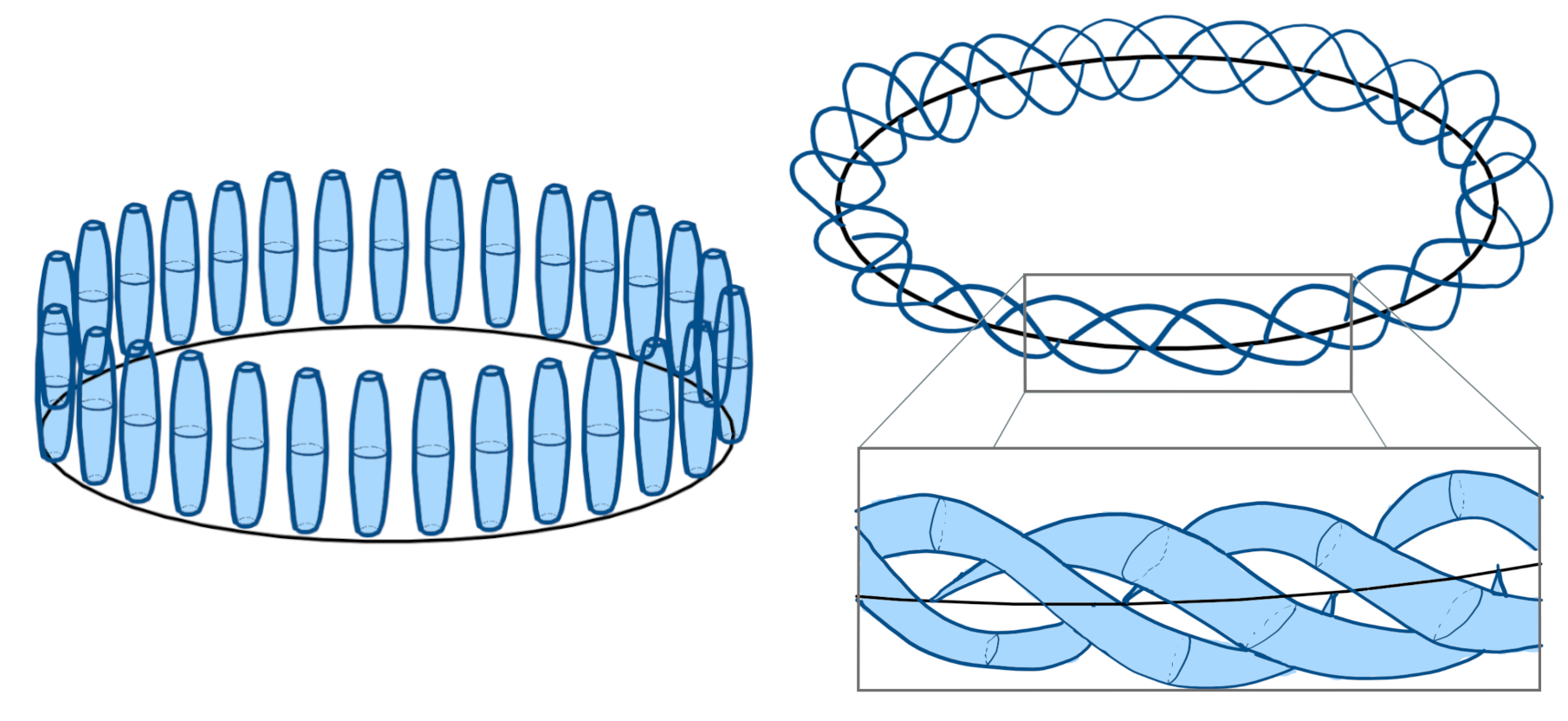}
\caption{Twisted $\P^1$-bundle over an abelian torsor \label{fig:singfibre2}}
\end{figure}

\subsection*{Towards Langlands duality for singular Hitchin fibres}
The Langlands duality of Higgs bundle moduli spaces is a reincarnation of mirror symmetry and its geometric interpretation in terms of integrable systems by the Strominger-Yau-Zaslow conjecture \cite{SYZ}. 
%
For Hitchin systems, mirror symmetry is connected to another important duality in pure mathematics - the so-called Langlands duality. For a algebraic group $G$ there exists a Langlands dual group $G^L$, such that conjecturally the representation theory of $G$ is controlled by Galois representations into $G^L$. 
%

Starting from the work of Hausel and Thaddeus \cite{HaTh03} for $G=\SL(n,\C)$, $G^L=\PSL(n,\C)$ and Hitchin \cite{Hi07} for $G=\Sp(2n,\C)$, $G^L=\SO(2n+1,\C)$ and $G=G^L=\mathsf{G}_2$, Donagi and Pantev \cite{DoPa12} established the following formulation of Langlands duality of $G$-Hitchin systems for a complex semi-simple Lie group $G$: 
\begin{itemize}
\item[i)] The Hitchin bases $B_G$ and $B_{G^L}$ are isomorphic and the isomorphism restricts to the regular loci $B_G^\reg$ and $B_{G^L}^\reg$.
\item[ii)] The regular fibres over corresponding points $b \in B_G^\reg$ and $b'\in B_{G^L}^\reg$ are abelian torsors over dual abelian varieties. 
\end{itemize}

Concerning Langlands duality for singular Hitchin fibres of $\sl(2)$-type, we have to take a closer look at the abelian part of the spectral data. For $G=\Sp(2n,\C)$, the spectral curve $\Sigma$ has an involutive deck transformation $\sigma: \Sigma \rightarrow \Sigma$. 
\begin{wrapfigure}{r}{0.4\textwidth} \vspace{-0.5cm}
\begin{tikzpicture}\matrix (m) [matrix of math nodes,row sep=2em,column sep=3em,minimum width=2em]
  { & & \tilde{\Sigma} \\
    \Sigma/\sigma  & \Sigma & \\
    & X & \\
  };
  \path[-stealth]
  	(m-1-3) edge (m-2-2)
  			edge  (m-2-1)
  			edge  (m-3-2)
    (m-2-2) 
           edge  (m-2-1)
           edge  (m-3-2)
    (m-2-1) edge (m-3-2);
\end{tikzpicture} \caption{Commutative diagram of spectral curves  \label{fig:spectralcurves}}
\end{wrapfigure}
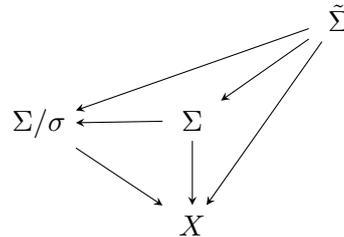 
The quotient defines a complex algebraic curve $\Sigma/\sigma$. Together with the normalised spectral curve $\tilde{\Sigma}$, we obtain the commutative diagram of spectral curves in Figure \ref{fig:spectralcurves}. 

By definition the spectral curve $\Sigma$ is of $\sl(2)$-type if and only if $\Sigma/\sigma$ is smooth. In this case, there is an abelian variety associated to the 2-sheeted branched covering of Riemann surfaces $\tilde{\Sigma} \rightarrow \Sigma/\sigma$, the so-called Prym variety. The abelian part of the spectral data for $G=\Sp(2n,\C)$ is a torsor over this Prym variety. 

For $G= \SO(2n+1,\C)$, the abelian part of the spectral data is a union of torsors over a quotient of the Prym variety by the finite group $\Z_2^{2g}$, where $g$ is the genus of $X$. This quotient can be identified with the dual abelian variety. We obtain the following formulation of Langlands correspondence for singular Hitchin fibres of $\sl(2)$-type.

\begin{coro}[Corollary \ref{coro:duality}]  Let $b \in B_{\Sp(2n,\C)}=B_{\SO(2n+1,\C)}$ be of $\sl(2)$-type, such that the spectral curve is irreducible and reduced. Then the Hitchin fibres $\H_{\Sp(2n,\C)}^{-1}(b)$ and $\H_{\SO(2n+1,\C)}^{-1}(b)$ are related as follows: 
\begin{itemize}
\item[i)] The abelian parts of the spectral data are unions of torsors over dual abelian varieties.
\item[ii)] The parameter spaces of Hecke transformations are isomorphic. 
\end{itemize}
\end{coro} 


\subsection*{Limiting configurations for singular Hitchin fibres}
Another recent development in the study of Higgs bundle moduli spaces is the analysis of the asymptotic of the hyperkähler metric. Evolving from an intriguing conjectural picture developed by Gaiotto, Moore and Neitzke \cite{GMN13}, it was shown that on the regular locus of the Hitchin map the asymptotic of the hyperkähler metric are described by a so-called semi-flat metric \cite{MSWW19,Fr18b,FMSW20}. 
This is a hyperkähler metric defined on any algebraically completely integrable system by the theory of special Kähler manifolds \cite{Fr99}. It does not extend over the singular locus, but Gaiotto, Moore and Neitzke suggest that it can be modified to define a hyperkähler metric on $\M_G$. Recent progress in this direction can be found in \cite{Tu19}. 

As a first step in analysing the asymptotic of the hyperkähler metric, Fredrickson, Mazzeo, Swoboda, Weiss and Witt studied limits of solutions to the Hitchin equation along rays to the ends of the moduli space \cite{MSWW16,Mo16,Fr18a}. It was shown in \cite{MSWW14,Fr18a}, that these so-called limiting configurations satisfy a decoupled version of the Hitchin equation and are completely determined by spectral data. In Theorem \ref{theo:limit:sl2type}, we will use the semi-abelian spectral data explained above to construct solutions to the decoupled Hitchin equation for $\sl(2)$-type fibres of the symplectic and odd orthogonal Hitchin system. We conjecture them to be limiting configurations. For $\SL(2,\C)$, this is a theorem by Mochizuki \cite{Mo16}.

\subsection*{Reader's guide}
The paper is structured into four sections. In Section \ref{sec:SP2nC}, we will introduce $\sl(2)$-type Hitchin fibres of the symplectic Hitchin system. We prove the identification of these Hitchin fibres with $\SL(2,\C)$-Hitchin fibres on $\Sigma/\sigma$ and give the parametrisation by semi-abelian spectral data using the results of \cite{Ho1}. In Section \ref{sec:SO}, we repeat these considerations for the odd orthogonal group. 

Summing up, we formulate the Langlands correspondence for $\Sp(2n,\C)$- and $\SO(2n+1,\C)$-Hitchin fibres of $\sl(2)$-type in Section \ref{sec:langlands}. Finally, in Section \ref{sec:lim}, we will construct solutions to the decoupled Hitchin equation and motivate why we conjecture theses to be limiting configurations.

\subsection*{Acknowledgement} This paper represents a substantial part of the author's PhD thesis. Thank you to Daniele Alessandrini for a great supervision and many useful comments and questions during the preparation of this work. Thank you to Anna Wienhard, Beatrice Pozzetti and Xuesen Na for enlightening conversations and their interest in this project. We thank the anonymous referee for a careful reading of the manuscript.

This work was supported by the Deutsche Forschungsgemeinschaft (DFG, German Research Foundation) [281869850 (RTG 2229)]; the Klaus Tschira Foundation; and the U.S. National Science Foundation [DMS 1107452, 1107263, 1107367 "RNMS: GEometric structures And Representation varieties" (the GEAR Network)].


\section{\texorpdfstring{$\sl(2)$}{sl(2)}-type fibers of symplectic Hitchin systems}\label{sec:SP2nC}
\subsection{The \texorpdfstring{$\Sp(2n,\C)$}{Sp(2n,C)}-Hitchin system}\label{sec:sp(2n,C)-Hit}
Let $X$ be a Riemann surface of genus $g \geq 2$ and let $M$ denote a holomorphic line bundle on $X$ with $\deg(M) > 0$. 
\begin{defi} An $M$-twisted $\Sp(2n,\C)$-Higgs bundle is a triple $(E,\Phi,\omega)$ of a 
\begin{itemize}
\item[i)] holomorphic vector bundle $E$ of rank $2n$, 
\item[ii)] an anti-symmetric bilinear form $\omega \in H^0(X,\bigwedge ^2 E^\vee)$, such that $\omega^{\wedge n} \in H^0(X,\det(E^\vee))$ is  non-vanishing, and
\item[iii)] $\Phi \in H^0(X,\End(E) \otimes M)$, such that $\omega(\Phi \, \cdot, \cdot)=-\omega(\cdot,\Phi\, \cdot)$. 
\end{itemize}
\end{defi}

\begin{theo}[Simplified stability condition \cite{GGM09}]
A $\Sp(2n,\C)$-Higgs bundle $(E,\Phi,\omega)$ is stable, if for all isotropic $\Phi$-invariant subbundles $0 \neq F \subsetneq E$
\[ \deg(F) < 0 .
\] 
\end{theo}

\noindent Let $\M_{\Sp(2n,\C)}(X,M)$ denote the moduli space of stable $M$-twisted $\Sp(2n,\C)$-Higgs bundles on $X$. This is a complex algebraic variety (see \cite{GGM09,Sch08}). For $M=K_X$, it is a complex symplectic manifold of dimension 
\[ (2g-2)(2n^2+n).
\] Let $A \in \sp(2n,\C)$. The characteristic polynomial of $A$ is of the form
\[ T^{2n}+a_2(A) T^{2n-2} + \dots + a_{2n}(A) \in \C[T].
\] The coefficients $(a_2,\dots,a_{2n})$ are homogeneous generators of $\C[\g]^G$ and the associated Hitchin map is given by
\begin{align*} \H_{\Sp(2n,\C)}: \quad \M_{\Sp(2n,\C)}(X,M) &\rightarrow B_{2n}(X,M):=\bigoplus\limits_{i=1}^n H^0(X,M^{2i}), \\
(E,\Phi) \quad &\mapsto \quad (a_2(\Phi), \dots , a_{2n}(\Phi)).
\end{align*} This is proper, surjective, flat, holomorphic map \cite{Ni91,Si95II}. For $M=K_X$, the Hitchin map restricted to a dense subset $B^{\mathsf{reg}}_{2n} \subset B_{2n}$ defines an algebraically completely integrable system \cite{Hi87b,Hi07}.

The characteristic equation of $(E,\Phi) \in \H_{\Sp(2n,\C)}^{-1}(a_2,\dots,a_{2n})$ is given by
\[ \eta ^{2n} + a_2\eta^{2n-2} + \dots + a_{2n} =0. 
\] Let $p_M: M \rightarrow X$ the bundle map, then $\eta$ can be interpreted as the tautological section $\eta: M \rightarrow p_M^*M$. The pointwise eigenvalues of the Higgs field form the a complex analytic curve
\[ \Sigma:=Z_{M}(\eta^{2n}+p_M^*a_2\eta^{2n-2}+\dots +p_M^*a_{2n}) \subset \mathsf{Tot}M.
\] This is the so-called spectral curve.
The projection $p_M$ restricts to a $2n$-sheeted branched analytic covering $\pi: \Sigma \rightarrow X$. Recall that in general the spectral curve is singular at the points, where different sheets meet. Due to the specific type of characteristic equation the spectral curve comes with an involutive automorphism $\sigma: \Sigma \rightarrow \Sigma$ reflecting in the zero section of $M$.

For $M=K_X$, the regular locus $B_{2n}^{\mathsf{reg}}$ is the subset of the Hitchin base, where the spectral curve $\Sigma$ is smooth. The fibres over $B_{2n}^{\mathsf{reg}}$ are torsors over the Prym variety 
\[ \Prym(\Sigma \rightarrow \Sigma/\sigma)
\] (see \cite{Hi87b,Hi07}). We will reprove this result in Theorem \ref{theo:Sp(2n):strat}.


The regular locus can be detected by the $\sp(2n,\C)$-discriminant. Consider the representation of $\sp(2n,\C)$ 
\[ \left\{ A \in \Mat(2n \times 2n,\C) \mid A^\tr J_{2n} + J_{2n}A=0 \right\}, \text{ where } J_{2n}=\begin{pmatrix} 0 & \id_n \\ -\id_n & 0
\end{pmatrix}.
\] A Cartan subalgebra is given by 
\[ \h= \left\{ H=\diag(h_1, \dots, h_n, -h_n, \dots, -h_1) \mid h_i \in \C
 \right\}.
\] 
Define $e_i \in \h^\vee$ by $e_i(H)=h_i$. Then a root system is given by 
\[ \Delta=\{ \pm e_i \pm e_j \mid 1\leq i<j \leq n \} \cup \{ \pm 2e_i \mid 1 \leq i \leq n \}.
\] The $\sp(2n,\C)$-discriminant is the invariant polynomial defined by product over all roots
\[ \disc_{\sp}:=\prod\limits_{\alpha \in \Delta} \alpha \in \C[\h]^W \cong \C[\g]^G. 
\] There are two types of roots differing by their length. The roots $\pm 2e_i$ have $\sqrt{2}$ times the length of the roots $\pm e_i \pm e_j$ (as depicted in the Dynkin diagram). The Weyl group $W$ preserves the inner product on $\h$ and hence the set of long/short roots. Therefore, we can define invariant polynomials in $\C[\g]^G$ by the product over the long/short roots. The product over the long roots $\prod_{i=1}^{n} -4e_i^2$ is (up to a scalar) the determinant function on $\h$. We refer to the product over the short roots as the reduced $\sp(2n,\C)$-discriminant
\[ \disc_{\sp}^{\textsf{red}}:=\prod_{i <j } -(e_i\pm e_j)^2.
\] We have
 \[ \disc_{\sp}=\det\disc_{\sp}^{\textsf{red}}.
 \]
The discriminant of a Higgs bundle $(E,\Phi) \in \M_{\Sp(2n,\C)}(X,M)$ is the section 
\[ \disc_{\sp}(\Phi) \in H^0(X,M^{2n^2}).
\] Being invariant polynomials $\disc_{\sp}$ and $\disc_{\sp}^{\textsf{red}}$ factor through the Hitchin map. For $\underline{a} \in B_{2n}(X,M)$, we will write $\disc_{\sp}(\underline{a})$ and $\disc_{\sp}^{\textsf{red}}(\underline{a})$ for the (reduced) discriminant computed in this manner. 
 
\begin{lemm}\label{lem:sp:discr} If all zeros of $\disc_{\sp}(\underline{a}) \in H^0(X,M^{2n^2})$ are simple, then the spectral curve is smooth.
\end{lemm}
\begin{proof}[Proof of Lemma \ref{lem:sp:discr}] Let $x \in X$ be a simple zero of  
 \[ \disc_{\sp}(\underline{a})= a_{2n}\disc_{\sp}^{\textsf{red}}(\underline{a}) \in H^0(X,M^{2n^2}).
 \] If $a_{2n}$ has a simple zero at $x$ and $\disc_{\sp}^{\textsf{red}}(\underline{a})(x) \neq 0$, then $\pi^{-1}(x) \in \Sigma$ contains a simple ramification point on the zero section. If $\disc_{\sp}^{\textsf{red}}(\underline{a})$ has a simple zero at $x$ and $a_{2n}(x) \neq 0$, then $\pi^{-1}(x) \in \Sigma$ contains two simple ramification points $0 \neq \lambda,-\lambda \in M_x$. Hence, the spectral curve is smooth.
\end{proof}

\begin{exam}[$\Sp(4,\C)$]\label{exam:disc:sp(4,C)} For $(a_2,a_4) \in B_4(X,M)$,
The $\sp(4,\C)$-discriminant is given by
\[ \disc_{\sp}(a_2,a_4)=a_4(a_2^2-4a_4).
\] If instead we compute the discriminant of the characteristic polynomial - the $\sl(4,\C)$-discriminant -, we obtain 
\[ \disc_{\sl(4,\C)}(a_2,a_4)=a_4(a_2^2-4a_4)^2.
\] This expression has higher order zeros for all $(a_2,a_4) \in B_4(X,M)$. Hence, the $\sl(4,\C)$-discriminant can not detect the regular locus of the $\Sp(4,\C)$-Hitchin map.
\end{exam}

\noindent \textbf{Notation} In the following we will often consider a branched covering of Riemann surfaces $p: Y \rightarrow X$. To avoid confusion, we will refer to points in $Y$, where different sheets meet or equivalently zeros of $\del p$ as ramification points and to the images of these points under $p$ as branch points. We denote by $R=\div(\partial p)\in \Div(Y)$ the ramification divisor and refer to its coefficient $R_y$ at a ramification point $y \in Y$ as the ramification index. $B:=\Nm(R) \in \Div(X)$ is referred to as branch divisor.

\vspace*{0.5cm}
\subsection{\texorpdfstring{$\sl(2)$}{sl(2)}-type spectral curves}\label{sec:sl(2)-type}
In this subsection, we will define the class of $\sl(2)$-type fibres of the $\Sp(2n,\C)$-Hitchin map. These Hitchin fibres are distinguished by the singularities of the spectral curve, such that for $G=\SL(2,\C)$ all Hitchin fibres are of $\sl(2)$-type.\\

\noindent \begin{minipage}[c]{0.69\textwidth}
Let $\underline{a} \in B_{2n}(X,M)$, $\Sigma \subset \mathrm{Tot}(M)$ the associated spectral curve and $\sigma$ the involutive biholomorphism reflecting in the zero section of $M$. Being the zero section of a polynomial with coefficients in a line bundle on a Riemann surface, the spectral curve $\Sigma$ is algebraic. The involution $\sigma$ defines an algebraic $\Z_2$-action on $\Sigma$. We will construct its quotient in the algebraic category. A geometric quotient by this action is given by
\[ \pi_2: \Sigma \rightarrow \Sigma/\sigma:= \mathsf{Spec}(\O_\Sigma^{\sigma}), 
\] where $\O_\Sigma^\sigma$ denotes the sheaf of $\sigma$-invariant regular functions on $\Sigma$. As $\pi$ is invariant under the $\Z_2$-action, we obtain the commutative diagram on the right of this paragraph.
\end{minipage}\hfill
\begin{minipage}[c]{0.29\textwidth}
\begin{tikzpicture}\matrix (m) [matrix of math nodes,row sep=3em,column sep=4em,minimum width=2em]
  {
     & \Sigma \\
    \Sigma/\sigma & \\ 
    & X \\
  };
  \path[-stealth]
    (m-1-2) 
           edge node [above] {$\pi_2$ } (m-2-1)
           edge node [right] {$\pi$} (m-3-2)
    (m-2-1) edge node [left] {$\pi_n$}(m-3-2);
\end{tikzpicture}
\end{minipage}
 
\begin{defi}\label{def:sl(2)-type_Sp} 
An element $\underline{a} \in B_{2n}(X,M)$ is called of $\sl(2)$-type, if $\Sigma/\sigma$ is smooth. In this case, $\H_{\Sp(2n,\C)}^{-1}(\underline{a})$ is called $\sl(2)$-type Hitchin fibre. An $\Sp(2n,\C)$-Higgs bundle is called of $\sl(2)$-type, if it is contained in an $\sl(2)$-type Hitchin fibre.
\end{defi}

\begin{exam}
\begin{itemize} 
\item[i)] Let $n=1$. Then $X\cong \Sigma/\sigma$ is smooth for all $a_2 \in H^0(X,M^2)$ and hence all Hitchin fibres are of $\sl(2)$-type.
\item[ii)] A regular point $\underline{a} \in B_{2n}^{\mathsf{reg}}(X,M)$ is of $\sl(2)$-type. In this case, $\Sigma$ is smooth and so is $\Sigma/\sigma$. The fibers are isomorphic to $\Prym( \Sigma \rightarrow \Sigma/\sigma)$, which in turn determines a regular Hitchin fibre of the $\pi_n^*K_X$-twisted $\SL(2,\C)$-Hitchin system on $\Sigma/\sigma$. 
\item[iii)] Consider $n=2$ and $(a_2,a_4) \in B_{4}(X,M)$, such that $\Sigma$ is smooth except of one point $p \in \Sigma$ on the zero section. Assume that the spectral curve is locally at $p$ isomorphic to $Z(y^2-z^2) \subset \C^2$ with $\sigma: \C^2 \rightarrow \C^2, (y,z) \mapsto (-y,z)$. Locally, the quotient $\Sigma/\sigma$ is isomorphic to the affine curve $\mathsf{Spec}((\C[y,z]/(y^2-z^2))^\sigma)$. There is an isomorphism
\[ \left(\C[y,z]/(y^2-z^2)\right)^\sigma \rightarrow \C[w], \quad y^2 \mapsto w^2, z \mapsto w
\] and hence $\Sigma/\sigma$ is smooth at $p$. In conclusion, $(a_2,a_4) \in B_{4}\setminus B_4^{\mathsf{reg}}$ is of $\sl(2)$-type. 
\end{itemize}
\end{exam}

\begin{prop}\label{prop:sp:descr:sl(2)-type_spectral_curve} A point $\underline{a} \in B_{2n}(X,M)$ is of $\sl(2)$-type if and only if all singular points of $\Sigma$ lie on the zero section of $M \rightarrow X$ and only two sheets meet in the singular points. In particular, all singular points of $\Sigma$ are of type $A_k$, $k \geq 1$, i. e. higher nodes and cusps. 

If $\disc_{\sp}^{\red}(\underline{a}) \in H^0(X,M^{2n(n-1)})$ has simple zero and $Z(a_{2n-2}) \cap Z(a_{2n})= \varnothing$, then $\underline{a}=(a_2, \dots, a_{2n}) \in B_{2n}(X,M)$ is of $\sl(2)$-type. 
\end{prop}
\begin{proof} If $\underline{a} \in B_{2n}(X,M)$ is of $\sl(2)$-type, there can not be any singular points away from the zero section of $M$. Otherwise $\Sigma/\sigma$ is singular, too. Let $y \in \Sigma$ be a singular point on the zero section. Choose a trivialization $M \rest_U \cong U \times \C$ over a coordinate neighbourhood $(U,z)$ centred at $\pi(y)$ and let $(z,\lambda)$ be the induced coordinate on $M$. Then $\Sigma$ is locally given by the equation
\[ q(z,\lambda):= \lambda^{2n}+\lambda^{2n-2}a_2(z) + \dots + a_{2n}(z)=0
\] with the involution given by $\sigma: (z,\lambda) \mapsto (z,-\lambda)$. Because $y=(0,0)$ is a singular point, we have 
\[ \frac{\del}{\del z} \rest_{(z,\lambda)=(0,0)} \, q=  \frac{\del}{\del \lambda} \rest_{(z,\lambda)=(0,0)} \, q =0.
\] Hence, $\frac{\del}{\del z} \rest_{z=0} \, a_{2n}=0$, i.e.\ $a_{2n}$ has a higher order zero at $z=0$. Now, $\Sigma/\sigma$ is locally given by the equation
\[ q^\sigma(\eta,z)=\eta^{n}+\eta^{n-1}a_2(z) + \dots + a_{2n}(z)=0
\] and smooth at $(0,0)$ by assumption . Therefore, 
\[ 0 \neq \frac{\del}{\del \eta} \rest_{(z,\eta)=(0,0)} \,  q^\sigma = a_{2n-2}(0).
\] In particular, $\lambda=0$ is a zero of $q(0,\lambda)$ of multiplicity $2$ and hence only two sheets meet in the singular point. 

Conversely, if a singular point $p$ lies on the zero section and two sheets of the covering $\pi$ meet there, then $\Sigma$ is locally given by a polynomial equation of the form $y^2-z^k=0$. Let $R=\C[y,z]/(y^2-z^k)$. The ring of invariant functions $R^\sigma$ is generated by $y^2$ and $z$. In particular, 
\[  R^\sigma \rightarrow \C[z], \quad y^2 \mapsto z^k, z \mapsto z
\] defines an isomorphism of coordinate rings. Hence, $\mathsf{Spec}(R^\sigma) \cong \C$ and the quotient is smooth. 

The discriminant condition implies that, away from the zero section, the only points, where different sheets meet, are smooth ramification points of ramification index 1. Furthermore, $Z(a_{2n-2}) \cap Z(a_{2n})= \varnothing$ implies that only two sheets meet at the zero section, in particular at the singular points. Hence, the spectral curve is of $\sl(2)$-type by the first criterion.
\end{proof}

\begin{rema} Nevertheless, there can be smooth ramification points of $\pi: \Sigma \rightarrow X$ of higher order on the zero section of $M$ for an $\sl(2)$-type spectral curve $\Sigma$. For $n=2$, an example is the spectral curve defined by $(0,a_4) \in B_{4}(X,M)$ with $a_4$ having simple zeros. 
\end{rema}

\begin{rema} An irreducible algebraic/analytic subset $Z \subset \C^n$ is a $C^1$-manifold in a neighbourhood of a point $p$ if and if only $Z$ is locally given by algebraic/analytic equations
\[ F_1(x_1,\dots,x_n)=0, \ \dots, \ F_k(x_1,\dots, x_n)=0,
\] such that $D(F_1,\dots,F_k)$ has maximal rank at $p$. The backwards implication follows from the implicit function theorem. For the converse see \cite[page 13]{Mi68}.
\end{rema}

\begin{prop}\label{prop:canon:red:spec} Let $p: M^2 \rightarrow X$ the bundle map and $\eta: M^2 \rightarrow p^*M^2$ the tautological section. Let $(a_2,\dots,a_{2n}) \in B_{2n}(X,M)$ be of $\sl(2)$-type. The reduced spectral curve $\Sigma/\sigma$ is the zero divisor of
\[ \eta^n+a_2 \eta^{n-1} + \dots + a_{2n-2} \eta + a_{2n} \in H^0(M^2,p^*M^{2n}).
\] In particular, $K_{\Sigma/\sigma} \cong \pi_n^* \left(M^{2n-2} \otimes K_X \right)$ and $\O(R) \cong \pi_n^*M^{2n-2}$, where $R \in \Div(\Sigma/\sigma)$ is the ramification divisor of $\pi_n: \Sigma/\sigma \rightarrow X$.
\end{prop}
\begin{proof}
The first assertion is clear from the proof of the previous proposition. It is easy to see, that $K_{M^2} \cong p_{M^2}^*(K_X \otimes M^{-2})$ and hence by the adjunction formula
\[ K_{\Sigma/\sigma} = \left( K_{M^2} \otimes p_{M^2}^*M^{2n}\right) \rest_{\Sigma/\sigma} = \pi_n^* \left(M^{2n-2} \otimes K_X \right).
\] The last assertion follows as $\O(R)=K_{\Sigma/\sigma}\otimes \pi_n^*K_X^{-1}$.
\end{proof}

\begin{wrapfigure}{r}{0.4\textwidth} \vspace{-0.5cm}
\begin{tikzpicture}\matrix (m) [matrix of math nodes,row sep=2em,column sep=3em,minimum width=2em]
  { & & \tilde{\Sigma} \\
  & & \\
    \Sigma/\sigma  & \Sigma & \\
    & & \\ 
    & X & \\
  };
  \path[-stealth]
  	(m-1-3) edge (m-3-2)
  			edge node [above] {$\tilde{\pi}_2$ } (m-3-1)
  			edge node [right] {$\tilde{\pi}$} (m-5-2)
    (m-3-2) 
           edge node [above] {$\pi_2$ } (m-3-1)
           edge node [right] {$\pi$} (m-5-2)
    (m-3-1) edge node [left] {$\pi_n$}(m-5-2);
\end{tikzpicture} \caption{Spectral curves  \label{fig2}}
\end{wrapfigure}
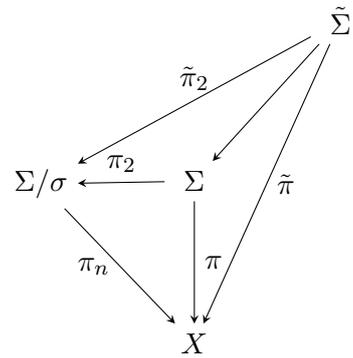 
In the subsequent analysis of $\sl(2)$-type Hitchin fibres, another version of the spectral curve plays an important role. We can naturally associate a smooth curve $\tilde{\Sigma}$ to the singular spectral curve $\Sigma$ by normalisation. It can be defined as the unique extension of the covering $\pi\rest_{\Sigma^\times}: \Sigma^\times \rightarrow X^\times$  to a holomorphic covering of Riemann surfaces. Here $\cdot^\times$ refers to the complement of ramification resp. branch points. If $\Sigma/\sigma$ is smooth, it can be defined in the same way as the extension of the covering of Riemann surfaces $\pi_2\rest_{\Sigma^\times}: \Sigma^\times \rightarrow (\Sigma/\sigma)^\times$. Intrinsically, it is the analytic curve $\tilde{\Sigma}$ associated to the integral closure of the structure sheaf. We obtain the commutative diagram in Figure \ref{fig2}. 

For $a_{2n} \in H^0(X,M^{2n})$, let
\begin{align*} 
n_{\odd}:=n_{\odd}(a_{2n}):=\# \{ x \in Z(a_{2n}) \mid x \text{ zero of odd order} \}.
\end{align*}

\begin{lemm} Let $(a_2,\dots,a_{2n}) \in B_{2n}(X,M)$ be of $\sl(2)$-type. Then the genus of $\Sigma/\sigma$ is given by
\[ g(\Sigma/\sigma)=n(g-1)+(n^2-n)\deg(M)+1.
\] The genus of the normalised spectral curve is 
\[ g(\tilde{\Sigma})=2n(g-1)+2(n^2-n)\deg(M) + \tfrac12 n_{\odd} +1.
\] If $M=K$, we have
\[ g(\Sigma/\sigma)=(2n^2-n)(g-1)+1 
\] and
\[ g(\tilde{\Sigma})=(4n^2-2n)(g-1) + \tfrac12 n_{\odd} +1.
\]
\end{lemm}
\begin{proof} This is immediate from Proposition \ref{prop:canon:red:spec} and the Riemann-Hurwitz formula.
\end{proof}

\subsection{\texorpdfstring{$\sl(2,\C)$}{sl(2,C)}-type Hitchin fibres are fibres of an \texorpdfstring{$\SL(2,\C)$}{SL(2,C)}-Hitchin map}\label{ssec:sp:fibreident}

In this subsection, we prove the main theorem in the $\Sp(2n,\C)$-case identifying the $\sl(2)$-type Hitchin fibres with fibres of an $\SL(2,\C)$-Hitchin system on the spectral curve $\Sigma/\sigma$. 

\begin{prop}\label{prop:pushforward:sp2nC} Let $p: Y \rightarrow X$ be an $s:1$ covering of Riemann surfaces. Fix a square root $\O(R)^{\frac12}$ of the ramification divisor $R\in \Div(Y)$. Let $(E, \Phi) \in \M_{\SL(2,\C)}(Y,p^*M)$, then the pushforward $(p_*(E\otimes \O(R)^{\frac12}),p_*\Phi)$ defines a $M$-twisted $\Sp(2s,\C)$-Higgs bundle on $X$. 

Recall that the Ramification divisor $R$ has even degree by the Riemann-Hurwitz formula.
\end{prop}

\begin{proof}
Let $E':=E\otimes \O(R)^{\frac12}$. The pushforward $p_*E'$ is locally free and 
\[ p_*\Phi: p_*E' \rightarrow p_*(E' \otimes p^*M)=p_*E'\otimes M
\] defines a $M$-twisted Higgs field on $p_*E'$. The symplectic form $\omega \in H^0(Y,\bigwedge^2E^\vee)$ induces a degenerate symplectic form $\omega'=\omega(\partial p)^{-1} \in H^0(Y,\bigwedge^2E^\vee(-R))$ on $E'$. Let $U \subset X$ be trivially covered, such that $E'\rest_{p^{-1}(U)}$ is trivial. Hence $p^{-1}(U)= V_1 \sqcup \dots \sqcup V_s$. Let $s_{ij}$ with $i=1,2; j=1,\dots, s$ be symplectic frames of $E' \rest_{V_j}$, i.e.\
\[ \omega'\rest_{V_j}= \begin{pmatrix} 0 & 1 \\ -1 & 0 \end{pmatrix}
\] respective $s_{1j},s_{2j}$. Then the induced symplectic form on $p_*(E')\rest_{U}$ is given by 
\[ p_* \omega' \rest_{U}= \left(\begin{array}{cc|c|cc}  0 & 1 &&& \\ -1 & 0 &&& \\ \hline & & \ddots & & \\ \hline &&&  0 & 1 \\ &&& -1 & 0
\end{array} \right)
\] respective the frame $s_{ij}$. This defines a symplectic form $p_*\omega'$ on $p_*E' \rest_{Y^\times}$, where $Y^\times= Y \setminus \supp R$. Obviously, $p_*\omega'( p_*\Phi \, \cdot, \cdot)=-p_*\omega'(\cdot, p_*\Phi \, \cdot)$. 

To extend the symplectic form over the branch points, we use a description of the algebraic pushforward by local $\Z_k$-invariant bundles at the corresponding ramification point. Let $\omega':=\omega(\partial p)^{-1} \in H^0(Y,\bigwedge^2 (E')^\vee)$. Let $y \in Y$ be a ramification point of order $k$. Choose coordinate neighbourhoods $(V,z)$ centred at $y$ and $(U,w)$ centred at $p(y)$, such that the projection map is given by $p: z \mapsto z^k$. Let $\xi$ a primitive root of unity of order $k$. Then $\tau: V \rightarrow V,\ z \mapsto \xi z$ induces a local $\Z_k$-action interchanging the sheets. Consider the local holomorphic $\Z_k$-vector bundle 
\[ F:= E'\rest_V \oplus \tau^*E'\rest_V \oplus \dots \oplus (\tau^{k-1})^*E'\rest_V.
\] Let $s_1,s_2$ be a symplectic frame of $E' \rest_V$, then 
\begin{align*}
 s_{ij}:=\tfrac{1}{k} (s_i + \xi^{j}\tau^*s_i + \xi^{2j}(\tau^2)^*s_i + \dots + \xi^{(k-1)j}(\tau^{k-1})^*s_i)
\end{align*} for $i \in \{ 1,2\}$ and $0 \leq j \leq k-1$ define a frame of $F$, such that the $\Z_k$-action is given by
\[ \diag(1, 1, \xi, \xi, \dots, \xi^{k-1}, \xi^{k-1}).
\]
The induced degenerate symplectic form $\Omega=\omega' + \tau^* \omega' + \dots +(\tau^{k-1})^*\omega'$  is given by 
\[ \Omega(s_{1l},s_{2m})=\left\{ \begin{array}{ll} z^{-k+1} & \text{for } l+m=k-1 \\ 
0 & \text{otherwise}.
\end{array} \right.
\] We obtain a local $\Z_k$-invariant holomorphic vector bundle $\hat{F}$ descending to $p(V)$ as a Hecke transformation
\[ 0 \rightarrow \hat{F} \rightarrow F \rightarrow \bigoplus\limits_{i=1}^{k-1} (\O_y/z^{i}\O_y)^2 \rightarrow 0
\] introducing the new transition function
\[ \psi_{01}=\diag(1,1,z,z, \dots, z^{k-1},z^{k-1})
\] respective the frame $s_{ij}$. The Hecke transformed Higgs bundle is $\Z_k$-invariant and descends to a local frame of the pushforward $p_*(E',\Phi)$ on $p(V)$. The induced symplectic form is given by 
\[ \hat{\Omega}=(\psi_{01}^*\Omega)(\hat{s}_{1l},\hat{s}_{2m})=\left\{ \begin{array}{ll} 1 & \text{for } l+m=k-1 \\ 
0 & \text{otherwise},
\end{array} \right.
\] where $\hat{s}_{ij}$ denotes the induced frame of $\hat{F}$ at $y$. Hence, $\hat{\Omega}$ descends to a non-degenerate symplectic form on $p_*E'$. Again it is clear that the induced Higgs field $p_*\Phi$ is anti-symmetric with respect to the symplectic form. 
\end{proof}

In the same way one proves:
\begin{prop}\label{prop:push:pairing} Let $\pi: Y \rightarrow X$ be a branched covering of Riemann surfaces. Let $E,F$ holomorphic vector bundles on $Y$ and $\beta: E \otimes F \rightarrow \C$ a non-degenerate bilinear pairing. Fix a square root $\O(R)^{\frac12}$. Then there is an induced non-degenerate pairing
\[ \pi_*(E \otimes \O(R)^{\frac12}) \otimes \pi_*(F \otimes \O(R)^{\frac12}) \rightarrow \C.
\]  
\end{prop}

Let $\underline{a} \in B_{2n}(X,M)$ be of $\sl(2)$-type. The spectral curve $\Sigma$ comes with a section $\lambda \in H^0(\Sigma,\pi^*M)$ solving the spectral equation. The product $\lambda \sigma^*(\lambda) \in H^0(\Sigma,\pi^*M^2)$ defines a $\sigma$-invariant section descending to $b_2 \in H^0(\Sigma/\sigma, \pi^*_n M^2)$. 

\begin{prop}\label{prop:pullback:toHit} Let $\underline{a} \in B_{2n}(X,M)$ be of $\sl(2)$-type and $b_2 \in H^0(\Sigma/\sigma,\pi_n^*M^2)$ the induced section. There is a holomorphic map
\[ \H^{-1}_{\Sp(2n,\C)}(\underline{a}) \rightarrow \H^{-1}_{\SL(2,\C)}(b_2) \subset \M_{\SL(2,\C)}(\Sigma/\sigma,\pi_n^*M).
\]   
\end{prop}

\begin{proof}
Let $(E,\Phi)\in \H^{-1}_{\Sp(2n,\C)}(\underline{a})$. The pullback of the characteristic polynomial along $\pi_n: \Sigma/\sigma \rightarrow X$
\[ \lambda^{2n}+ \pi_n^*a_2\lambda^{2n-2} + \dots + \pi_n^*a_{2n} 
\] factors through $\lambda^2+b_2$ and hence defines a generalised locally free eigen sheaf $E_2$ by
\[ 0 \rightarrow E_2 \rightarrow \pi^*_n E \xrightarrow{\pi_n^*\Phi^2 +b_2 \id} \pi^*_n (E \otimes M^2) \rightarrow E_2 \otimes \pi^*_n M^{2n} \rightarrow 0. \]
Here the cokernel of $\pi_n^*\Phi^2 +b_2 \id$ is identified with $E_2 \otimes \pi^*_n M^{2n}$ using the symplectic form. The dualized exact sequence tensored with $\pi_n^*M^2$ results in
\[ 0 \rightarrow E_2^\vee \otimes \pi^*_n M^{2-2n} \rightarrow \pi^*_n E^\vee \xrightarrow{(\pi_n^*\Phi^2 +b_2 \id)^\vee} \pi^*_n (E^\vee \otimes M^2) \rightarrow E_2^\vee \otimes \pi_n^*M^2 \rightarrow 0.
\] The symplectic form $\omega$ identifies $E$ with $E^\vee$ and from the anti-symmetry of the Higgs field the bundle map $\pi_n^*\Phi^2 +b_2 \id_{\pi_n^*E}$ is self-dual. Hence, there is an induced isomorphism $E_2 \cong E_2^\vee \otimes \pi^*_n M^{2-2n}$. In particular, $\omega$ restricts to a symplectic form $\omega_2$ on $E_2 \otimes \pi^*_n M^{n-1}$ and the induced Higgs field $\Phi_2$ on $E_2$ is anti-symmetric with respect to it. Hence, $(E_2,\Phi_2)$ is a $\pi_n^*M$-twisted $\SL(2,\C)$-Higgs bundle on $\Sigma/\sigma$. Stability will be discussed in the proof of the following theorem.
\end{proof}

\begin{theo}\label{Theo:isom:Hitchinfibres} Let $\underline{a} \in B_{2n}(X,M)$ be of $\sl(2)$-type and $b_2 \in H^0(\Sigma/\sigma,\pi_n^*M^2)$ the induced section. The holomorphic map
\[ \H^{-1}_{\Sp(2n,\C)}(\underline{a}) \rightarrow \H^{-1}_{\SL(2,\C)}(b_2) 
\] defined in Proposition \ref{prop:pullback:toHit} is a biholomorphism. Its inverse is given by Proposition \ref{prop:pushforward:sp2nC}.
\end{theo}
\begin{proof}
We need to show that the holomorphic maps defined in Proposition \ref{prop:pushforward:sp2nC}  and \ref{prop:pullback:toHit} with $\O(R)^{\frac12}=\pi_n^*M^{n-1}$ are inverse to each other. Let $(E_2,\Phi_2) \in \H^{-1}(b_2)$. By Proposition \ref{prop:pushforward:sp2nC}, $(\pi_{n*}(E_2\otimes \pi_n^*M^{n-1}),\pi_{n*}\Phi_2)$ defines a $\Sp(2n,\C)$-Higgs bundle on $X$ with spectral curve $\Sigma$. We have a natural map 
\[ E_2\otimes \pi_n^*M^{1-n} \rightarrow E_2\otimes \pi_n^*M^{n-1}
\] by multiplying with the canonical section of $\O(R) \cong \pi_n^*M^{2n-2}$. This induces an inclusion
\[ \iota: E_2\otimes \pi_n^*M^{1-n} \rightarrow \pi_n^*\pi_{n*}(E_2\otimes \pi_n^*M^{n-1}).
\] It is clear by construction that the $\im(\iota)=\Ker(\pi_n^*\pi_{n*}\Phi_2^2 +b_2 \id)$. 

For the converse, let $(E,\Phi) \in \H^{-1}_{\Sp(2n,\C)}(\underline{a})$ and denote by $(E_2,\Phi_2)$ the induced $\SL(2,\C)$-Higgs bundle on $\Sigma/\sigma$. It is clear that 
\[ \pi_{n*}(E_2\otimes \pi_n^*M^{n-1}, \Phi_2)\rest_{X^\times} \cong (E,\Phi)\rest_{X^\times},
\] where $X^\times=X\setminus \pi_n(\supp R)$. We are left with showing that this isomorphism extends over the branch points. Let $x \in X$ be a branch point. For simplicity of notation we assume that it corresponds to a ramification point $y \in Y$ of index $n-1$. Let $(V,z)$ resp. $(U,w)$ be coordinate neighbourhoods centred at $y$ resp. $x$, such that the covering is given by $\pi_n: V \rightarrow U, z \mapsto z^n$. We have a local automorphism $\tau: V \rightarrow V, z \mapsto \xi z$, where $\xi$ is a primitive $n$-th root of unity. This automorphism interchanging the sheets induces a local $\Z_k$-action on $\Sigma/\sigma$ at $y$. The pullback $\pi_n^*(E,\Phi) \rest_V$ is invariant by this $\Z_k$-action. 
As explained in the proof of \ref{prop:pushforward:sp2nC}, we can obtain a frame of $\pi_n^*\pi_{n*}(E_2\otimes \pi_n^*M^{n-1},\Phi_2)\rest_{\pi_n^{-1}X^\times}$ at $y$ by extending
\[ (F,\Psi)= (E',\Phi)\rest_{V^\times} \oplus \tau^*(E',\Phi)\rest_{V^\times} \oplus \dots \oplus (\tau^{k-1})^*(E',\Phi)\rest_{V^\times}
\] to a $\tau$-invariant $\SL(2n,\C)$-Higgs bundle at $y$. This is the unique way to do so. Hence, the isomorphism extends over the branch points.

Finally, let us check that this isomorphism preserves stability. If $\Sigma$ is irreducible, there are no Higgs field-invariant subbundles of $E_2$ or $E$ and hence all Higgs bundles in the corresponding $\SL(2,\C)$- resp. $\Sp(2n,\C)$-Hitchin fiber are stable. So let us assume $\Sigma$ is reducible. Being $\sl(2)$-type the spectral curve has two irreducible components $\Sigma=\Sigma_1 \cup \Sigma_2$ interchanged by $\sigma$. Let $L \subset E_2$ be an $\Phi_2$-invariant line bundle, then $V=\pi_{n*}L \otimes\pi_n^*M^{n-1}$ is a $\pi_{n*}\Phi_2$-invariant isotropic subbundle of $E=\pi_{n*}E_2 \otimes\pi_n^*M^{n-1}$ of degree
\begin{align} \deg(V)=\deg\left(\Nm(L) \otimes M^{n(n-1)} \otimes \det(\pi_{n*} \O_{\Sigma/\sigma})\right) = \deg(L), \label{equ:stabl:deg}
\end{align} where we used that 
\begin{align*} \det(\pi_{n*} \O_{\Sigma/\sigma})^2=\O(-B)=\Nm(-R)=M^{2n(1-n)}.
\end{align*} Hence, if $(E,\Phi,\omega)$ is stable, $(E_2,\Phi_2)$ is stable. Furthermore, all $\Phi$-invariant subbundles of $(E,\Phi,\omega)$ are of this form. (In other words, there are two of them corresponding to the irreducible components of $\Sigma$.) Hence, the converse holds true as well.
\end{proof}

\subsection{Semi-abelian spectral data for \texorpdfstring{$\sl(2)$}{sl(2)}-type Hitchin fibres}\label{sec:sp:semi-abel}\ \\
In this section, we apply the results of \cite{Ho1} to $\Sp(2n,\C)$-Hitchin fibres of $\sl(2)$-type. Let us start by defining the twisted Prym varieties, the abelian part of the spectral data.

\begin{defi} Let $p: Y \rightarrow X$ be branched covering of Riemann surfaces. Let $N \in \Pic(X)$. Define
\[ \Prym_N(p):= \Nm_p^{-1}(N),
\] where $\Nm_p: \Pic(Y) \rightarrow \Pic(X)$ is the norm map associated to $p$.
\end{defi}

\begin{lemm} $\Prym_N(p)$  is an abelian torsor over the Prym variety $\Prym_{\O_X}(p) = \Ker(\Nm_p)$, whenever it is non-empty. If $p: Y \rightarrow X$ is two-to-one and $\sigma: Y \rightarrow Y$ the involution interchanging the sheets, then 
\[ \Prym_N \subseteq \{ L \in \Pic(Y) \mid L \otimes \sigma^*L=p^*N\}.
\] If $p$ is not unbranched, this is an equality.
\end{lemm}
\begin{proof} The first statement is clear. For the second, see \cite{Ho1} Proposition 5.6.
\end{proof}

In the same vein as in \cite{Ho1} for $\SL(2,\C)$, the semi-abelian spectral data will define a stratification of the singular Hitchin fibers. The strata are indexed by so-called Higgs divisors.

\begin{defi} Let $a_{2n} \in H^0(X,K_X^{2n})$. An associated Higgs divisor is a divisor $D \in \Div(X)$, such that $\supp (D) \subset Z(a_{2n})$ and for all $x \in Z(a_{2n})$ 
\[ 0 \leq D_x \leq \tfrac{1}{2} \ord_x(a_{2n}).
\]
\end{defi}

\begin{lemm}\label{lemm:local_form+Higgs_divisor}
Let $\underline{a} \in B_{2n}(X,K_X)$ be of $\sl(2)$-type. Let $(E,\Phi) \in \H^{-1}_{\Sp(2n,\C)}(\underline{a})$ and $x \in Z(a_{2n}) \subset X$ a zero of order $m$. There exists a coordinate neighbourhood $(U,z)$ centred at $x$ and a frame of $E\rest_U$, such that the Higgs field is given by
\[ \Phi=\left( \begin{array}{cc|c} 0 & z^{l_x} &\\ z^{m-l_x} & 0 & \\ \hline & & \phi \end{array}
\right) \d z
\] for some $0 \leq l_x \leq \tfrac{m}{2}$. Here $\phi$ has pointwise non-zero eigenvalues. The Higgs divisor of $(E,\Phi)$ is the divisor
\[ D= D(E,\Phi)=\sum_{x \in Z(a_{2n})} l_x.
\]
\end{lemm}
\begin{proof}
By assumption 0 is an eigenvalue of $\Phi_x$ of algebraic multiplicity two. Therefore, we can find a coordinate neighbourhood $(U,z)$ centred at $x$, such that $(E,\Phi)\rest_U=(E_0 \oplus E_1, \Phi_0 \oplus \Phi_1)$, where $E_0$ is of rank $2$ with $\Phi_0(x)$ nilpotent and $E_1$ is of rank $2n-2$ with $\Phi_1$ having non-zero eigenvalues. Moreover, by the anti-symmetry of $\Phi$ the symplectic form $\omega$ restricts to a symplectic form on $E_0$ and $E_1$. Now, we can bring $(E_0,\Phi_0)$ in the desired form by \cite{Ho1} Lemma 5.1. 
\end{proof}

For $a_{2n} \in H^0(X,K_X^{2n})$, let
\begin{align*} n_{\even}:=\# \{ x \in Z(a_{2n}) \mid x \text{ zero of even order} \}, \\
n_{\odd}:=\# \{ x \in Z(a_{2n}) \mid x \text{ zero of odd order} \}.
\end{align*} For $D \in \Div^+(X)$ a Higgs divisor associated to $a_{2n}$, let 
\[ n_{\diag}(D):=\#\{ x \in Z(a_{2n}) \mid D_x= \tfrac12 \ord_x(a_{2n}) \}.
\] 

\begin{theo}\label{theo:Sp(2n):strat} Let $\underline{a} \in B_{2n}(X,K_X)$ be of $\sl(2)$-type, such that $\Sigma$ is irreducible and reduced. 
There is a stratification 
\[ \H_{\Sp(2n,\C)}^{-1}(\underline{a})= \bigsqcup_{D} \S_D
\] by locally closed analytic sets $\S_D$ indexed by Higgs divisors associated to $a_{2n}$. If $a_{2n}$ has at least one zero of odd order, every stratum $S_D$ is a holomorphic fiber bundle 
\[ (\C^*)^{r}\times (\C)^{s} \rightarrow \S_D \rightarrow \Prym_{\pi_n^*K_X^{-1}(D)}(\tilde{\pi}_2)
\] with
\[ r=n_{\even}-n_{\diag}(D), \quad r+s=2n(g-1)  - \deg(D)-\frac{n_{\odd}}{2}.
\] 
If all zeros of $a_{2n}$ have even order, $\tilde{\pi}_2$ is unbranched and each stratum $\S_D$ is a $2:1$-branched covering of a holomorphic $(\C^*)^{r}\times (\C)^{s}$-bundle over 
\[ \Prym_{I\pi_n^*K_X^{-1}(D)}(\tilde{\pi}_2).
\] with $r,s$ given by above formulae. Here, $I$ denotes the unique non-trivial line bundle on $\Sigma/\sigma$, such that $\tilde{\pi}_2^*I=\O_{\tilde{\Sigma}}$. In both cases, 
\[ \dim \S_D = (2n^2+n)(g-1)-\deg(D).
\]
\end{theo}
\begin{proof}
This is a direct consequence of Theorem \ref{Theo:isom:Hitchinfibres} and the stratification result for singular fibres of $\SL(2,\C)$-Hitchin systems with irreducible and reduced spectral curve in \cite{Ho1} Theorem 5.13. The dimension of the twisted Prym varieties is given by 
\[ \dim \Prym(\tilde{\pi}_2)=g(\tilde{\Sigma}) - g(\Sigma/\sigma)=n(2n-1)(g-1)+ \frac{n_{\odd}}{2}.
\] 
\end{proof}

\begin{theo}\label{theo:Sp(2n):global_fibreing} Let $\underline{a}=(a_2, \dots, a_{2n}) \in B_{2n}(X,K_X)$ be of $\sl(2)$-type, such that
$a_{2n}\in H^0(X,K_X^{2n})$ has only zeros of odd order. Then $\H_{\Sp(2n,\C)}^{-1}(\underline{a})$ is a holomorphic fiber bundle over $\Prym_{\pi_n^*K_X^{-1}}(\tilde{\pi}_2)$ with fibres given by the compact moduli of Hecke parameters described in \cite{Ho1} Section 7. 
\end{theo}
\begin{proof}
This is a direct consequence of \cite{Ho1} Theorem 7.13.
\end{proof}

Putting together \cite{Ho1} Corollary 7.14, 7.16 and Example 8.3, 8.5 we obtain:
\begin{exam}\label{exam:sp(2n,C)_first_deg} Let $\underline{a}=(a_2, \dots, a_{2n}) \in B_{2n}(X,K_X)$ be of $\sl(2)$-type. Let $a_{2n}$ have $k_l$ zero of order $l$ for $l \in \{2,3,4,5\}$ and at least one zero of odd order. Then up to normalisation $\H_{\Sp(2n,\C)}^{-1}(\underline{a})$ is given by a holomorphic fibre bundle 
\[ (\P^1)^{k_2+k_3} \times (\P(1,1,2))^{k_4+k_5} \rightarrow \H_{\Sp(2n,\C)}^{-1}(\underline{a}) \rightarrow \Prym_{\pi_n^*K_X^{-1}}(\tilde{\pi}_2).
\]
\end{exam}

\begin{theo}\label{theo:fibrebdle_trivial} The fiber bundles over abelian varieties appearing in Theorem \ref{theo:Sp(2n):strat}, \ref{theo:Sp(2n):global_fibreing} and Example \ref{exam:sp(2n,C)_first_deg} are smoothly trivial.
\end{theo}
\begin{proof}
This will be proved in Section \ref{sec:lim} using analytic techniques. 
\end{proof}

\begin{coro}\label{coro:irred_comp} Let $\underline{a}=(a_2, \dots, a_{2n}) \in B_{2n}(X,K_X)$ of be $\sl(2)$-type, such that
$a_{2n}\in H^0(X,K_X^{2n})$ has at least one zero of odd order, then $\H_{\Sp(2n,\C)}^{-1}(\underline{a})$ is an irreducible complex space. If all zero of $a_{2n}$ have even order, then $\H_{\Sp(2n,\C)}^{-1}(\underline{a})$ is connected and has four irreducible components.
\end{coro}
\begin{proof}
This follows from \cite{Ho1} Corollary 8.6 and Theorem 8.8.
\end{proof}

\begin{rema} Notice that the identification of Hitchin fibres in Theorem \ref{Theo:isom:Hitchinfibres} is not restricted to $\sl(2)$-type Hitchin fibres with irreducible and reduced spectral curve. In particular, the parametrization of singular Hitchin fibres with reducible  spectral curve in \cite{GO13} Section 7 describes certain $\sl(2)$-type Hitchin fibres of the $\Sp(2n,\C)$-Hitchin system, for all $n \in \N$.
\end{rema}


\section{\texorpdfstring{$\sl(2)$}{sl(2)}-type fibres of odd orthogonal Hitchin systems} \label{sec:SO}
\subsection{The \texorpdfstring{$\SO(2n+1,\C)$}{SO(2n+1,C)}-Hitchin system}
Let $G=\SO(2n+1,\C)$ and 
\begin{align*} \so(2n+1,\C)= \left\{ A \in ´\Mat(n \times n , \C) \mid A^{\tr}J_{2n+1} + J_{2n+1}A=0  \right\},
\end{align*} where 
\[ J_{2n+1}=\begin{pmatrix} 0 & \id_n & 0 \\ \id_n & 0 & 0 \\ 0 & 0 & 1 \end{pmatrix}.
\]
Then a Cartan subalgebra is given by 
\[ \h= \{H=\diag(h_1, \dots, h_n,-h_1,\dots,-h_n,0) \mid h_i \in \C \}.
\]
Define by $e_i \in \h^\vee$ by $e_i(H)=h_i$. Then a root system is given by 
\[ \Delta=\{ \pm e_i \pm e_j \mid 1\leq i,j \leq n, i \neq j \} \cup \{ \pm e_i \mid 1 \leq i \leq n \}.
\] As before the $\so(2n+1,\C)$-discriminant decomposes by the length of the roots
\[ \disc_{\so}=\prod\limits_{i=1}^{n} -e_i^2 \disc_{\so}^{\red}, \quad \text{where} \quad \disc_{\so}^{\red}=\prod_{i \neq j } -(e_i\pm e_j)^2.
\] The characteristic polynomial of $A \in \so(2n+1,\C)$ has the form
\[ \lambda(\lambda^{2n} + a_2 \lambda^{2n-2} + \dots + a_{2n}).
\] The coefficients $a_2, \dots, a_{2n}$ form a basis of the invariant polynomials $\C[\g]^G$. 

\begin{defi}
An $M$-twisted $\SO(m,\C)$-Higgs bundle is a triple $(E,\Phi,\omega)$ of a 
\begin{itemize}
\item[i)] holomorphic vector bundle $E$ of rank $m$ with $\det(E) \cong \O_X$,
\item[ii)] a holomorphic non-degenerate symmetric bilinear form $\omega \in H^0(X,S^2 E^\vee)$, and
\item[iii)] a Higgs field $\Phi \in H^0(X,\End(E) \otimes M)$, such that $\omega(\Phi \, \cdot, \cdot)=-w(\cdot,\Phi\, \cdot)$. 
\end{itemize}
$(E,\Phi,\omega)$ is called stable, if for all isotropic $\Phi$-invariant subbundles $0 \neq F \subsetneq E$
\[ \deg(F) < 0.
\] (see \cite{GGM09} for this simplified stability condition).
\end{defi}
Let $\M_{\SO(m,\C)}(X,M)$ be the moduli space of stable $M$-twisted $\SO(m,\C)$-Higgs bundles on $X$. For $m=2n+1$ the Hitchin map is given by
\begin{align*}
\H_{\SO(2n+1,\C)}: \M_{\SO(2n+1,\C)}(X,M) &\rightarrow B_{\SO(2n+1,\C)}(X,M):=\bigoplus\limits_{i=1}^n H^0(X,M^{2i}), \\
(E,\Phi,\omega) \quad &\mapsto \quad (a_2(\Phi), \dots , a_{2n}(\Phi)).
\end{align*} In particular, we observe that $B_{\SO(2n+1,\C)}(X,M)=B_{2n}(X,M)$. Let $(E,\Phi,\omega) \in \H_{\SO(2n+1,\C)}^{-1}(a_2, \dots, a_{2n})$, then the characteristic polynomial of $\Phi$ is given by 
\[ \lambda(\lambda^{2n} + a_{2n-2} \lambda^{2n-2} + \dots + a_{2n}).
\] Hence, the spectral curve decomposes in two irreducible components $\underline{0} \cup \Sigma$, where $\underline{0}$ is the image of the zero section in $M$ and $\Sigma$ is the $\Sp(2n,\C)$-spectral curve associated to $(a_2,\dots,a_n)$.

\begin{defi}\label{def:sl(2)-type_SO} An element of the Hitchin base $\underline{a} \in B_{\SO(2n+1,\C)}(X,M)$ is called of $\sl(2)$-type, if $\Sigma/\sigma$ is smooth. In this case, the corresponding Hitchin fibre $\H^{-1}_{\SO(2n+1,\C)}(\underline{a})$ is called of $\sl(2)$-type. A $M$-twisted $\SO(2n+1,\C)$-Higgs bundles is of $\sl(2)$-type, if it is contained in a $\sl(2)$-type Hitchin fibre.
\end{defi}

From Lemma \ref{lem:sp:discr} and Proposition \ref{prop:sp:descr:sl(2)-type_spectral_curve}, we immediately have
\begin{lemm}\label{lem:so(2n+1):discr} Let $\underline{a} \in B_{2n}(X,M)$. If all zeros of $\disc_{\so}(\underline{a}) \in H^0(X,M^{2n^2})$ are simple, then $\Sigma$ is smooth. If $\disc_{\so}^{\red}(\underline{a}) \in H^0(X,M^{2n(n-1)})$ has simple zeros, then $\underline{a}$ is of $\sl(2)$-type.
\end{lemm}

Hence, the descriptions and properties of $\sl(2)$-type spectral curves in Section \ref{sec:sl(2)-type} carry over to $\sl(2)$-type Hitchin fibres of the odd orthogonal Hitchin system by adding the irreducible component $\underline{0}$. 

\subsection{Odd orthogonal \texorpdfstring{$\sl(2,\C)$}{sl(2,C)}-type fibres as fibres of an \texorpdfstring{$\SO(3,\C)$}{SO(3,C)}-Hitchin map}\label{ssec:so:identfibre}

\begin{lemm}\label{lem:so:local} Let $(E,\Phi, \omega) \in \M_{\SO(2n+1,\C)}(X,M)$ be of $\sl(2)$-type. Let $p \in Z(\det(\Phi))$ be a zero of order $m$, then there exists a coordinate neighbourhood $(U,z)$ centred at $p$ and an orthogonal splitting $(E,\Phi)\rest_U=(V_0 \oplus V_1, \Phi_0 \oplus \Phi_1)$, such that $V_0$ is of rank $3$ and $\Phi_0(p)$ is nilpotent, and $V_1$ is of rank $2n-2$ containing the eigenspaces to eigenvalues $\lambda$ with $\lambda(p) \neq 0$. There exists a orthogonal frame of $V_0\rest_U$, such that 
\[ \Phi_0(z)= z^{l_p} \begin{pmatrix} 0 & 1-z^{m-2l_p} & 0 \\ z^{m-2l_p}-1 & 0 & i(z^{m-2l_p}+1) \\ 0 & -i(z^{m-2l_p}+1) & 0 \end{pmatrix}  \d z.
\]
\end{lemm}
\begin{proof}
By construction $(V_0,\Phi_0)$ is a $\mathsf{O}(3,\C)$-Higgs bundle on $U$. Due to the exceptional isomorphism $\SO(3,\C) \cong \PSL(2,\C)$ the Higgs field $\Phi_0$ can be obtained as $\ad(\Psi)$ for a $\SL(2,\C)$-Higgs field $\Psi$ (cf. Section \ref{sec:langlands}). By \cite{Ho1} Lemma 5.1, we can find a local frame, such that 
\[ \Psi = \begin{pmatrix} 0 & z^{l_p} \\ z^{m-l_p} & 0 \end{pmatrix} \d z.
\] With respect to the induced local frame of $V_0$ the Higgs field $\Phi$ is given by 
\[ \Phi= \ad(\Psi)=\begin{pmatrix} 0 & -z^{l_p} & 0 \\ -z^{m-l_p} & 0 & z^{l_p}\\ 0 & z^{m-l_p} & 0 \end{pmatrix} \d z
\] and the orthogonal structure induced by the Killing form by
\[ \begin{pmatrix} && 1 \\ &1 & \\ 1 && \end{pmatrix}.
\] Choosing an orthogonal frame we obtain the desired form. 
\end{proof}

\begin{defi} Let $(E,\Phi, \omega) \in \M_{\SO(2n+1,\C)}(X,M)$ be of $\sl(2)$-type. The Higgs divisor of $(E,\Phi, \omega)$ is the divisor 
\[ D(E,\Phi,\omega):= \sum_{p \in Z(a_{2n})} l_p,
\] where $l_p$ is defined by the previous lemma. 
\end{defi}

\begin{lemm}\label{lemm:so:kernel_Hecketrafo}
Let $(E,\Phi,\omega) \in \M_{\SO(2n+1,\C)}(X,M)$ be of $\sl(2)$-type and $D$ its Higgs divisor, then 
\begin{itemize} \item[i)]$\Ker(\Phi) \cong M^{-n}(D)$ and $\omega\rest_{\Ker(\Phi)}=\frac{a_{2n}}{s_D^2} \in H^0(X,M^{2n}(-2D))$, where $s_D$ denotes the canonical section of $\O(D)$.
\item[ii)] there is an exact sequence of coherent sheaves
\[ 0 \rightarrow \O(\Ker(\Phi) \oplus \Ker(\Phi)^{\perp}) \rightarrow \O(E) \rightarrow \mathcal{T} \rightarrow 0,
\] where $\mathcal{T}$ is a torsion sheaf with $\det(\mathcal{T}) \cong \O(\Lambda-2D)$. 
\item[iii)] $(E,\Phi,\omega)$ is uniquely determined by $D$ and 
\[ \left(\ker(\Phi)^\perp,\Phi \rest_{\Ker(\Phi)^{\perp}}, \omega \rest_{\Ker(\Phi)^{\perp}} \right).
\] 
\end{itemize}
\end{lemm}
\begin{proof} \begin{itemize}[wide=\parindent] \item[i)]
The proof of the first assertion is closely following an argument in Section 4.1/4.2 of \cite{Hi07} using the local form for the Higgs field describe in Lemma \ref{lem:so:local}. Let $x \in X$ and $(U,z)$ a coordinate chart centred at $x$. Consider an orthogonal splitting $E\rest_U=V_0 \oplus V_2 \oplus \dots \oplus V_n$, such that $V_0$ is as in the lemma and $V_i$ for $i \geq 2$ is rank $2$ containing the eigen spaces to eigenvalues $\pm \lambda_i \neq 0$. Let $e_0,e_1,e_2$ be an orthogonal frame for $V_0$, such that $\Phi_0$ has the form described in the Lemma and $e_{2i-1},e_{2i}$ an orthogonal frame of $V_i$ of eigen sections of $\Phi$. Then the induced alternating bilinear form $\alpha:=\omega(\Phi \cdot, \cdot)$ is given by
\[ \ \hspace{1.1cm} \alpha=i z^l \left(e_2 \wedge (e_3 +i e_1) + z(\cdots)\right) + i \lambda_2 (e_3 \wedge e_4)+ \dots+ i \lambda_n (e_{2n-1} \wedge e_{2n} ) .
\] Let us assume that with respect to our frame the volume form is given by $\vol=e_0 \wedge \dots \wedge e_{2n} \in H^0(U,\det(E))$. Then, we can write $\bigwedge^{n} \alpha \in H^0(U,\bigwedge^{2n}E \otimes M^{n})$ as a contraction $i_{v_0} \vol$ with 
\[ v_0 = -i^{n-1} z^{l} \lambda_2 \cdots \lambda_n (e_3 +i e_1)+ z^{l+1}(\cdots) \in H^0(U,E \otimes M^n).
\] So $v_0$ defines a non-vanishing section of $H^0(X,E \otimes M^n(-D))$ that spans the kernel of $\Phi$. Hence, $\Ker:=\Ker(\Phi) \cong M^{-n}(D)$.

Furthermore, using the local form of the previous lemma one computes that for $p \in Z(a_{2n})$ we have $\omega \rest_{\ker}=z^{\ord_p a_{2n}-2D_p}$. Hence (up to the right choice of $s_D$) 
\[ \omega \rest_{\ker} =\frac{a_{2n}}{s_D^2} \in H^0(X,M^{2n}(-2D)).
\]
\item[ii)] $\ker^\perp \subset E$ is a $\Phi$-invariant subbundle of rank $2n$, such that 
\[ E\rest_U \cong \ker \oplus \ker^\perp \rest_U
\] for all open $U \subset X$, such that $U \cap Z(a_{2n}) = \varnothing$. Hence, the inclusions define an exact sequence of coherent sheaves
\[ 0 \rightarrow \O(\Ker \oplus \Ker^\perp) \rightarrow O(E) \rightarrow \mathcal{T} \rightarrow 0
\] with $\mathcal{T}$ a torsion sheaf supported on $Z(a_{2n})$. Now, $\det(\mathcal{T})$ can be computed from the local description in Lemma \ref{lem:so:local}. 
\item[iii)] Stated differently ii) tells us that $E$ is a Hecke modification of $\Ker \oplus \Ker^\perp$ (see \cite{Ba10} Definition 1.1). We need to show that there is a unique Hecke modification doing the job, i. e. a unique Hecke modification, such that 
\[ F= \Ker \oplus \Ker^\perp
\] with its degenerate symmetric bilinear form 
\[ \beta= \omega \rest_{\Ker} \oplus \omega \rest_{\Ker^\perp}
\] is transformed into an $\SO(2n+1,\C)$-bundle $(\hat{F},\hat{\beta})$. At $p \in Z(a_{2n})$ we have an orthogonal decomposition 
\[ \left(\ker^\perp, \Phi \rest_{\ker^\perp} \right)\rest_U=\left( V_2 \oplus V_{2n-2}, \Phi_2 \oplus \Phi_1 \right)
\] by restricting the orthogonal decomposition in Lemma \ref{lem:so:local}. One the one side, $V_2$ is of rank 2 and $\Phi_2(p)$ is nilpotent, on the other, $\Phi_1$ has non-zero eigenvalues and $\omega\rest_{V_1}$ is non-degenerate. Thereby, we are left with showing that we can find a unique Hecke modification twisting
\[ \left( \Ker\rest_{U} \oplus V_2 , \frac{a_{2n}}{s_D^2} \oplus \omega\rest_{V_2} \right)
\] into a $\SO(3,\C)$-bundle.

Using the local description of the Higgs field in Lemma \ref{lem:so:local} one can show that there are local frames $e_0$ of $\Ker_U$ and $e_1,e_2$ of $V_2$, such that the non-degenerate bilinear form at $p$ is given by 
\[  \frac{a_{2n}}{s_D^2} \oplus \omega\rest_{V_2} = \begin{pmatrix} z^{m-2l} & 0 & 0 \\ 0 & z^{m-2l} & 0 \\ 0 &0  & 1 \end{pmatrix},
\] where $m= \ord_p(a_{2n})$ and $l=D_p$. Hence, the Hecke modification can be assumed to take place in $\mathsf{span}\{e_0,e_1\}$. If there were two Hecke modifications,
\begin{center}
\begin{tikzpicture} \matrix (m) [matrix of math nodes,row sep=2em,column sep=3em,minimum width=2em]
  {
    \Ker\rest_{U} \oplus V_2 & \hat{F}_1 \\ 
    & \hat{F}_2 \\
  };
  \path[-stealth]
    (m-1-1) 
           edge node [above] {$s_1$ } (m-1-2)
           edge node [left] {$s_2$ \ } (m-2-2);
\end{tikzpicture}
\end{center}
such that $\hat{F}_1,\hat{F}_2$ are $\SO(3,\C)$-bundles with the induced orthogonal structure, then up to choosing frames $s_1 \circ s_2^{-1}$ reduces to a meromorphic $\SO(2,\C)$-gauge (an element of the $\SO(2,\C)$-loop group). It is not hard to show, that such a gauge is automatically holomorphic. Hence, the resulting $\SO(3,\C)$-bundles $\hat{F}_1,\hat{F}_2$ are isomorphic.
\end{itemize}
\end{proof}

\begin{prop}\label{prop:SO:pushforward} Let $\underline{a}\in B_{2n}(X,M)$ be of $\sl(2)$-type and $b_2 \in H^0(\Sigma/\sigma,\pi_n^*M^2)$ the induced section. The pushforward induces a holomorphic map
\[ \M_{\SO(3,\C)}(\Sigma/\sigma,\pi_n^*M) \supset \H_{\SO(3,\C)}^{-1}(b_2) \rightarrow \H_{\SO(2n+1,\C)}(\underline{a}) \subset \M_{\SO(2n+1,\C)}(X,M).
\]
\end{prop}
\begin{proof}
Let $(E_3,\Phi_3,\omega_3) \in \H_{\SO(3,\C)}^{-1}(b_2)$. The pushforward 
\[ \pi_{n*}\left(\ker(\Phi_3)^\perp \otimes \pi_n^*M^{n-1}, \Phi_3\rest_{\ker(\Phi_3)^\perp}, (\del \pi_n^{-1})\omega_3 \rest_{\ker(\Phi_3)^\perp} \right)
\] defines a $M$-twisted $\GL(2n,\C)$-Higgs bundle on $X$ with 
\[ \det\left( \pi_{n*}\left(\ker(\Phi_3)^\perp \otimes \pi_n^*M^{n-1}\right)\right) = M^{-n}(\Nm D).
\] and a symmetric bilinear form $\pi_{n*}\left((\del \pi_n^{-1})\omega_3 \rest_{\ker(\Phi_3)^\perp}\right)$, which is non-degenerate away from $Z(a_{2n})$ by Proposition \ref{prop:push:pairing}. Furthermore, $\pi_{n*} \Phi_3$ is anti-symmetric with respect to this bilinear form. 
Moreover, we have a induced Higgs divisor given by $\Nm(D)$ supported at $Z(a_{2n})$. Now there is a unique way to recover a $\SO(2n+1,\C)$-Higgs bundle $(E,\Phi,\omega)$ by Lemma \ref{lemm:so:kernel_Hecketrafo}. 

This reconstruction seems to depend on the Higgs divisor $D$. However, we saw in the proof of Lemma \ref{lemm:so:kernel_Hecketrafo} that the construction is local and only depends on the rank 2 subbundle of $\pi_{n*}(\ker(\Phi_3)^\perp \otimes \pi_n^*M^{n-1})$, on which the Higgs field has vanishing eigenvalue. So for a trivially covered neighbourhood $U \subset X$ of $x \in Z(a_{2n})$ it recovers $(E_3,\Phi_3) \rest_{V}$, where $V \subset \pi_n^{-1}(U)$ is the unique connected component, such that $\lambda \in H^0(V,\pi_n^*K)$ has a zero. Hence, $(E,\Phi,\omega)$ varies holomorphically with $(E_3,\Phi_3,\omega_3)$. 

Finally, we show that this map preserves stability. If $\Sigma$ is irreducible, there are no $\Phi$-invariant isotropic subbundles of $(E,\Phi,\omega)$. Hence, it is automatically stable. If this is not the case, being of $\sl(2)$-type the corresponding $\Sp(2n,\C)$-spectral curve decomposes into two irreducible components $\Sigma=\Sigma_1 \cup \Sigma_2$. The $\Phi_3$-invariant isotropic subbundles $L_1,L_2 \subset \ker(\Phi_3)^\perp \subset E_3$ are the eigen line bundles corresponding to the irreducible components $\Sigma_1, \Sigma_2$ of $\Sigma$. Their pushforwards $\pi_{n*}L_i$ define $\pi_{n*} \Phi_3$-invariant isotropic subbundles of $E$. These are all $\Phi$-invariant isotropic subbundles of $E$. Now, Equation \ref{equ:stabl:deg} in the proof of Theorem \ref{Theo:isom:Hitchinfibres}. Shows that $(E_3,\Phi_3,\omega_3)$ is stable if and only if $(E,\Phi,\omega)$ is. 
\end{proof}

\begin{prop}\label{prop:SO:pullback} Let $\underline{a}\in B_{2n}(X,M)$ be of $\sl(2)$-type and $b_2 \in H^0(\Sigma/\sigma,\pi_n^*M^2)$ the induced section. The pullback along $\pi_n: \Sigma/\sigma \rightarrow X$ induces a holomorphic map
\[ \H^{-1}_{\SO(2n+1,\C)}(\underline{a}) \rightarrow \H^{-1}_{\SO(3,\C)}(b_2) \subset \M_{\SO(3,\C)}(\Sigma/\sigma,\pi_n^*M) 
\]   
\end{prop}
\begin{proof}
Let $(E,\Phi,\omega)\in \H^{-1}_{\SO(2n+1,\C)}(\underline{a})$. The pullback of the characteristic polynomial to $\Sigma/\sigma$ 
\[ \lambda\left(\lambda^{2n}+ \pi_n^*a_2\lambda^{2n-2} + \dots + \pi_n^*a_{2n}\right)
\] factors through $\lambda(\lambda^2+b_2)$ and hence defines a generalised eigen bundle $E_3$ on $\Sigma/\sigma$ by
\[ 0 \rightarrow E_3 \rightarrow \pi^*_n E \xrightarrow{\Psi} \pi^*_n (E \otimes M^3) \rightarrow E_3 \otimes \pi^*_n M^{2n+1} \rightarrow 0, \] where 
\[ \Psi:=\pi_n^*\Phi\left(\pi_n^*\Phi^2 +b_2 \id_{\pi_n^*E}\right).
\] Here the cokernel of $\Psi$ is identified with $E_3 \otimes \pi^*_n M^{2n+1}$ using the orthogonal form. The dual exact sequence tensored with $\pi_n^*M^3$ results in
\[ 0 \rightarrow E_3^\vee \otimes \pi^*_n M^{2-2n} \rightarrow \pi^*_n E^\vee \xrightarrow{\Psi^\vee} \pi^*_n (E^\vee \otimes M^3) \rightarrow E_3^\vee \otimes \pi_n^*M^3 \rightarrow 0.
\] The orthogonal bilinear form $\omega$ identifies $E$ with $E^\vee$ and from the anti-symmetry of the Higgs field $\Psi^\vee=-\Psi$ under this identification. Hence, $\omega$ induces an isomorphism $E_3 \cong E_3^\vee \otimes \pi^*_n M^{2-2n}$. Finally, $\omega$ restricts to a symmetric, non-degenerate bilinear form $\omega_3$ on $E_3 \otimes \pi^*_n M^{n-1}$ and the induced Higgs field $\Phi_3$ on $E_3$ is anti-symmetric with respect to it. Hence, $(E_3,\Phi_3)$ is a $\pi_n^*M$-twisted $\SO(3,\C)$-Higgs bundle on $\Sigma/\sigma$. It will become clear that this map preserves stability by the proof of the following theorem stating that it is the inverse map to the one defined in Proposition \ref{prop:SO:pushforward}.
\end{proof}

\begin{theo}\label{theo:so:biholo} Let $\underline{a} \in B_{2n}(X,M)$ be of $\sl(2)$-type and let $b_2 \in H^0(\Sigma/\sigma, \pi_n^*M^2)$ the induced section. The holomorphic map 
between the Hitchin fibres
\[ \H^{-1}_{\SO(3,\C)}(b_2) \rightarrow \H^{-1}_{\SO(2n+1,\C)}(\underline{a})
\]  defined in Proposition \ref{prop:SO:pushforward} is a biholomorphism of complex spaces.
\end{theo}

\begin{proof}
We are left with showing that the maps defined in the previous propositions are inverse to each other. 
We start with $(E_3,\Phi_3) \in \H^{-1}_{\SO(3,\C)}(b_2)$. Consider the holomorphic map 
\[ \Ker(\Phi_3)^\perp \otimes \pi_n^*M^{1-n} \rightarrow \Ker(\Phi_3)^\perp \otimes \pi_n^*M^{n-1}
\] tensoring with $s_R= \del \pi_n \in H^0(\Sigma/\sigma, \pi_n^* M^{2n-2})$. This induces an embedding of locally free sheaves
\[ 0 \rightarrow \Ker(\Phi_3)^\perp \otimes \pi_n^*M^{1-n} \rightarrow \pi_n^* \pi_{n*} \left( \Ker(\Phi_3)^\perp \otimes \pi_n^*M^{n-1} \right).
\] By construction its image is $\Ker(\pi_n^*\pi_{n*}\Phi_3^2-b_2 \id)$. Hence, we recover $\Ker(\Phi_3)^\perp$ by the map defined in Proposition \ref{prop:SO:pullback}. This uniquely determines $(E_3,\Phi_3,\omega_3)$ by Lemma \ref{lemm:so:kernel_Hecketrafo} iii).  

For the converse, let $(E,\Phi,\omega) \in \H^{-1}_{\SO(2n+1,\C)}(\underline{a})$. Then 
\[ (E_3, \Phi_3) = \left( \Ker\Psi, \Phi\rest_{\Ker \Psi} \right)
\] decomposes $\pi^*(E,\Phi)$ into rank 3 subbundles. For $U \subset X$, such that 
\[ \pi_n^{-1}(U)= \bigsqcup_{i=1}^n U_i,
\] these are the generalised eigenbundles to the eigenvalues $0, \pm \lambda\rest_{U_i}$. The pushforward of $\ker(\Phi_3)^\perp \otimes \pi_n^*M^{n-1} \subset E_3$ reassembles the eigenbundles to $\pm \lambda\rest_{U_j}$ for all $i$. By Lemma \ref{lemm:so:kernel_Hecketrafo}
this uniquely determines a $\SO(2n+1,\C)$-Higgs bundle. Hence, we recover $(E,\Phi,\omega)$. 
\end{proof}

\begin{rema}[Hitchin's approach to regular fibers]
Another way to attack to problem is trying to generalise Hitchin's approach in \cite{Hi07}. Hitchin describes the regular $\SO(2n+1,\C)$-Hitchin fibres by relating them to the corresponding $\Sp(2n,\C)$-Hitchin fibre on $X$. Let $(V,\Phi,g) \in \H^{-1}_{\SO(2n+1,\C)}(\underline{a})$ with $\underline{a}$ of $\sl(2)$-type. Adopting Hitchin's notation, let $V_0 \subset V$ be the kernel line bundle and $\Phi': V/V_0 \rightarrow V/V_0$ the induced Higgs field. It is easy to see that $\alpha:=g( \Phi' \cdot , \cdot )$ defines a holomorphic anti-symmetric bilinear form on $V/V_0$ that is non-degenerate, where $\Phi$ has distinct eigenvalues. If $\deg(D) \equiv 0 \mod 2n$, where $D=D(V,\Phi)$, we can choose a square root $L^{2n}=K_X^{-n}(D)$ and define a symplectic Higgs bundle by 
\[ (E:= V /V_0 \otimes L, \phi', \alpha).
\] $\bigwedge ^{n} \alpha \in H^0(X,\det(E))$ is generically non-zero and $\det(E)= \O_X$ by Lemma \ref{lemm:so:kernel_Hecketrafo} i). Hence, $\alpha$ is non-degenerate on $E$. For regular Hitchin fibres, $D$ is always zero and therefore this defines a map
\begin{align*} \H^{-1}_{\SO(2n+1,\C)}(\underline{a}) \rightarrow \H^{-1}_{\Sp(2n,\C)}(\underline{a}).
\end{align*} Hitchin uses this map to study the regular $\SO(2n+1,\C)$-fibres as covering spaces of symplectic Hitchin fibres. The singular fibres are stratified by the Higgs divisors $D$. One the open and dense stratum, we have $D=0$ and we could apply the same argument. But for the lower strata $\deg(D)$ mod $2n$ is unconstrained. Hence, this trick does not generalise. 
\end{rema}


\section[Langlands correspondence]{Langlands correspondence for \texorpdfstring{$\sl(2)$}{sl(2)}-type Hitchin fibres}\label{sec:langlands} In this section, we compare the $\sl(2)$-type Hitchin fibres for the Langlands dual groups $\Sp(2n,\C)$ and $\SO(2n+1,\C)$ projection to the same point in the Hitchin base. Concerning the abelian part of the spectral data we will recover torsors over dual abelian varieties. This reproves and generalizes the result for regular fibres in \cite{Hi07}. The non-abelian part of the spectral data will not change under the duality. This is a new phenomena. We will start with the rank 1 case.

For $\rk(\g)=1$, we can compare the Hitchin fibres by using the exceptional isomorphisms $\Sp(2,\C) \cong \SL(2,\C)$ and $\SO(3,\C) \cong \PGL(2,\C)$. The moduli space of $\PGL(2,\C)$-Higgs bundles can be constructed as follows (see \cite{Ha13}). First recall that 
\[ \M_{\GL(1,\C)}(X,M) \cong \Pic(X) \times H^0(X,M)
\] is an abelian group with an action on $\M_{\GL(2,\C)}(X,M)$. Let
\[ (L,\lambda) \in \M_{\GL(1,\C)}(X,M)\quad  \text{and} \quad (E,\Phi) \in \M_{\GL(2,\C)}(X,M),
\] then the proper, holomorphic action is given by
\[  \left( (L,\lambda),(E,\Phi)\right) \mapsto (E \otimes L, \Phi + \lambda \id_E ).
\] We define the $\PGL(2,\C)$-Higgs bundle moduli as the orbifold quotient
\[ \M_{\PGL(2,\C)}(X,M) = \M_{\GL(2,\C)}(X,M) / \M_{\GL(1,\C)}(X,M).
\] Acting with $H^0(X,M)$, we can find a representative for each $\PGL(2,\C)$-Higgs bundles with $\tr(\Phi)=0$. Hence,
\[ \M_{\PGL(2,\C)}(X,M) \cong \H_{\GL(2,\C)}^{-1}(B_{\SL(2,\C)}(X,M)) / \Pic(X), 
\] where we think of $B_{\SL(2,\C)}(X,M) \subset B_{\GL(2,\C)}(X,M)$ by the obvious inclusion. For $N \in \Pic(X)$ define
\[ \M_{\SL(2,\C)}^N(X,M) = \left\{ (E,\Phi) \in M_{\GL(2,\C)}(X,M) \mid \det(E)=N, \tr(\Phi)=0 \right\}. 
\] The action of $\Pic(X)$ identifies $\M_{\SL(2,\C)}^{N_1}(X,M)$ and $\M_{\SL(2,\C)}^{N_2}(X,M)$, whenever $\deg(N_1)=\deg(N_2) \mod 2$. Hence, fixing a line bundle $N \in \Pic(X)$ of degree 1, we have
\begin{align} \M_{\PGL(2,\C)}(X,M)=\left( \M_{\SL(2,\C)}^{\O_X}(X,M) \sqcup \M_{\SL(2,\C)}^{N}(X,M) \right) / \Jac(X)[2], \label{equ:MPGL(n,C)}
\end{align} where $\Jac(X)[2] \cong \Z^{2g}_2$ denotes the group of two-torsion points of $\Jac(X)$. 

The isomorphism to the moduli space of $\SO(3,\C)$-Higgs bundles is defined using the adjoint representation  
\begin{align*} \M_{\PGL(2,\C)}(X,M) \quad  &\rightarrow \qquad \M_{\SO(3,\C)}(X,M) \\
(E,\Phi) \qquad &\mapsto  \left( (E \times_{\Ad} \sl(2,\C)) \otimes \det(E)^{-1}, \ad(\Phi), \omega \right).
\end{align*}
Here, the orthogonal structure $\omega$ is induced by the Killing form on $\sl(2,\C)$. Topologically $\SO(3,\C)$-Higgs bundles on a Riemann surface are classified by the second Stiefel-Whitney class 
\[ \mathsf{sw}_2 \in H^2(X,\Z_2) \cong \Z_2.
\] This is the obstruction to lift a $\SO(3,\C)$-Higgs bundle to a $\mathsf{Spin}(3,\C) \cong \SL(2,\C)$-Higgs bundle. Hence, under the isomorphism $\M_{\SL(2,\C)}^{\O_X}(X,M)/\Jac(X)[2]$ is map\-ped onto the connected component of $\SO(3,\C)$-Higgs bundles with $sw_2=0$ and $\M_{\SL(2,\C)}^{N}(X,M)/\Jac(X)[2]$ onto the connected component with $sw_2=1$. 

The Hitchin map 
\[ \H_{\PGL(2,\C)}:\M_{\PGL(2,\C)}(X,M) \rightarrow H^0(X,M^2)
\] is defined in terms of the decomposition (\ref{equ:MPGL(n,C)}) by the $\SL(2,\C)$-Hitchin map on each connected component.

For $(E,\Phi) \in \M_{\PGL(2,\C)}(X,M)$, there is a well-defined $\SL(2,\C)$-Higgs field $\Phi$ by (\ref{equ:MPGL(n,C)}). In particular, we can define a Higgs divisor $D(E,\Phi)$ as we did in Lemma \ref{lemm:local_form+Higgs_divisor}.

\begin{theo}\label{theo:SO(2n+1)_strat} Let $a_2 \in H^0(X,M^2)$, such that the spectral curve is irreducible and reduced, then there is a stratification 
\[ \H^{-1}_{\PGL(2,\C)}(a_2) = \bigsqcup_{D } \S_D
\] by finitely many locally closed analytic sets $\S_D$ indicated by Higgs divisors $D$ associated to $a_2$. If there is at least on zero of $a_2$ of odd order, each stratum is a holomorphic $(\C^*)^{r} \times \C^{s}$-bundle over
\[  \left( \Prym_{M^{-1}(D)}(\tilde{\pi}) \sqcup \Prym_{NM^{-1}(D)}(\tilde{\pi}) \right) /\Jac(X)[2], 
\] where 
\[ r=n_{\even}-n_{\diag}(D), \quad r+s=2n(g-1)  - \deg(D)-\frac{n_{\odd}}{2}.
\]
If all zeros of $a_{2}$ are of even order, each stratum $\S_D$ is a holomorphic  $(\C^*)^{r} \times \C^{s}$-bundle over 
\[  \left( \Prym_{IM^{-1}(D)}(\tilde{\pi}) \sqcup \Prym_{INM^{-1}(D)}(\tilde{\pi}) \right) /\Jac(X)[2], 
\] with $r,s$ given by above formulae. Here $I$ is the unique non-trivial line bundle on $X$, such that $\tpst I=\O_X$.
A local trivialisation of the fibre bundle $\S_D \subset \H^{-1}_{\PGL(2,\C)}(a_2)$ induces a local trivialisation of the fibre bundle structure of the corresponding stratum $\S_D \subset \H^{-1}_{\SL(2,\C)}(a_2)$ and vice versa.
\end{theo}

\begin{proof}
Fix a $\SL(2,\C)$-representative $(E,\Phi)$ of a Higgs bundle in 
\[ \S_D \subset \H^{-1}_{\PGL(2,\C)}(a_2) \subset \left( \M_{\SL(2,\C)}^{\O_X}(X,M) \sqcup \M_{\SL(2,\C)}^{N}(X,M) \right) / \Jac(X)[2]
\]  By \cite{Ho1} Theorem 5.5, we can associate an eigen line bundle $L$ on the normalised spectral cover $\tilde{\pi}: \tilde{\Sigma} \rightarrow X$ to $(E,\Phi)$. If $\det(E)=\O_X$, it will lie in $\Prym_{M^{-1}(D)}(\tilde{\pi})$ and, if $\det(E)=N$, in $\Prym_{NM^{-1}(D)}(\tilde{\pi})$. After choosing frames $s$ of $L$ at $\tilde{\pi}^{-1}Z(a_2)$ the $\SL(2,\C)$-Higgs bundle $(E,\Phi)$ is uniquely determined by its $u$-coordinate in $(\C^*)^{r} \times \C^{s}$ with $r,s$ as in the Theorem. The action by $\Jac(X)[2]$ lifts to the normalised spectral curve and induces an action
\begin{align*} \Jac(X)[2] \times \Prym_F(\tilde{\pi}) \rightarrow \Prym_F(\tilde{\pi}), \qquad (J,L) \mapsto \tpst J \otimes L
\end{align*} for $F \in \Pic(X)$. For $F=M^{-1}(D)$ and $F=NM^{-1}(D)$, this is exactly the action on the eigen line bundle induced by the action of $\Jac(X)[2]$ on $(E,\Phi)$. 

Recall, that in $\SL(2,\C)$-case for $a_2 \in H^0(X,M^2)$ having only zeros of even order, each stratum is a two-sheeted covering of a fibre bundle over the twisted Prym variety. This was due to the identification of $(E,\Phi)$ and $(E\otimes I,\Phi)$ via pullback. However, $I \in \Jac(X)[2]$ and so 
\[ \tpst: \M_{\PGL(2,\C)}(X,M) \rightarrow \M_{\PGL(2,\C)}(\tilde{\Sigma},\tpst M)
\] is injective.

The non-abelian part of the spectral data decodes the local Hecke parameter at $\tilde{\pi}^{-1}Z(a_2)$ and does not change under the action of $J \in \Jac(X)[2]$ on $(E,\Phi)$. Choosing a collection of frames $j$ of $J$ at $Z(a_2)$, we obtain a frame of $\tilde{\pi}^*J \otimes L$ at $\tilde{\pi}^{-1}Z(a_2)$ by $\tpst j \otimes s$. The $u$-coordinate does not depend on the choice of $j$ by \cite{Ho1} Proposition 5.8. This proves the last assertion.
\end{proof}

\begin{theo}\label{theo:SO(2n+1)_global_fibreing}
Let $a_2 \in H^0(X,M^2)$, such that the spectral curve is locally irreducible, then the $\PGL(2,\C)$-Hitchin fibre over $a_2$ is itself a holomorphic fibre bundle over 
\[ \left( \Prym_{M^{-1}(D)}(\tilde{\pi}) \sqcup \Prym_{NM^{-1}(D)}(\tilde{\pi}) \right) /\Jac(X)[2] 
\] with fibres given by the compact moduli of Hecke parameters. 
\end{theo}
\begin{proof}
This is a direct consequence of the previous theorem and \cite{Ho1} Theorem 7.13.
\end{proof}

\begin{exam}\label{exam:so(2n+1,C)_first_deg} Example \ref{exam:sp(2n,C)_first_deg} carries over to the $\PGL(2,\C)$-case. Let $a_{2n}$ have $k_l$ zeros of order $l$ for $l \in \{2,3,4,5\}$ and at least one zero of odd order. Then up to normalisation $\H_{\PGL(2,\C)}^{-1}(a_2)$ is given by a holomorphic  
\[ (\P^1)^{k_2+k_3} \times (\P(1,1,2))^{k_4+k_5}-
\] bundle over 
\[ \left( \Prym_{M^{-1}(D)}(\tilde{\pi}) \sqcup \Prym_{NM^{-1}(D)}(\tilde{\pi}) \right) /\Jac(X)[2].
\]
\end{exam}

Before we formulate the Langlands duality in $\rk(\g)=1$, let us identify the abelian part of the spectral data for $\PGL(2,\C)$ as an abelian torsor over the dual abelian variety to the Prym variety. 
\begin{prop}[\cite{HaTh03} Lemma 2.3]\label{prop:dual:prym} Let $\pi: Y \rightarrow X$ a $s$-sheeted covering of Riemann surfaces, then 
\[ \Prym(\pi)^\vee \cong \Prym(\pi)/\Jac(X)[s].
\]
\end{prop}

\begin{coro}\label{coro:duality_rank1} Let $a_2 \in H^0(X,M^2)$, such that the spectral curve is irreducible and reduced. The Hitchin fibres $\H^{-1}_{\PGL(2,\C)}(a_2)$ and $\H^{-1}_{\SL(2,\C)}(a_2)$ are related as follows:
\begin{itemize} 
\item[i)]  The abelian part of the spectral data are torsors over dual abelian varieties.
\item[ii)] The complex spaces of Hecke parameters are isomorphic.
\end{itemize}
\end{coro}
\begin{proof}
Assertion i) is immediate from the previous theorems and proposition. We showed in Theorem \ref{theo:SO(2n+1)_strat} that a trivialisation of the bundle of Hecke parameters of $\H^{-1}_{\SL(2,\C)}(a_2)$ induces a trivialisation of the bundle of Hecke parameters of $\H^{-1}_{\PGL(2,\C)}(a_2)$. The identity with respect to corresponding trivialisation induces an isomorphism between the complex spaces of Hecke parameters. 
\end{proof}

By Theorem \ref{theo:so:biholo}, these results carry over to higher rank. 

\begin{theo}\label{theo:so(2n+1):fibreing} Let $\underline{a} \in B_{2n}(X,K_X)$ be of $\sl(2)$-type with irreducible and reduced $\Sp(2n,\C)$-spectral curve. Fix $N \in \Pic(\Sigma/\sigma)$ of degree $1$. All the results from the previous section carry over to the $\SO(2n+1,\C)$-case. 

In explicit, there is a stratification 
\[ \H_{\SO(2n+1,\C)}^{-1}(\underline{a}) = \bigsqcup_{D} \S_D
\] by fibre bundles over disjoint unions of abelian torsors  indicated by Higgs divisors as described in Theorem \ref{theo:SO(2n+1)_strat}. If $a_{2n} \in H^0(X,K_X^{2n})$ has at least one zero of odd order, the disjoint union of abelian torsors is given by
\[ \left( \Prym_{\pi^*_nK_X^{-1}(D)}(\tilde{\pi}) \sqcup \Prym_{N\pi^*_nK_X^{-1}(D)}(\tilde{\pi}) \right) /\Jac(X)[2] 
\] If all zeros of $a_{2n}$ are of even order, it is
\[ \left( \Prym_{I\pi^*_nK_X^{-1}(D)}(\tilde{\pi}) \sqcup \Prym_{IN\pi^*_nK_X^{-1}(D)}(\tilde{\pi}) \right) /\Jac(X)[2],
\] where $I \in \Jac(\Sigma/\sigma)$ is the unique non-trivial line bundle, such $\tilde{\pi}_2^*I=\O_{\tilde{\Sigma}}$.

When $a_{2n}$ has only zeros of odd order, we obtain a global fibreing of the $\SO(2n+1,\C)$-Hitchin fibre over this union of abelian torsors as described in Theorem \ref{theo:SO(2n+1)_global_fibreing}.
Replacing the union of abelian torsors by the above, Example \ref{exam:so(2n+1,C)_first_deg} describes the first degenerations of singular $\sl(2)$-type Hitchin fibres for $\SO(2n+1,\C)$ up to normalisation. 
\end{theo}
\begin{proof}
This is immediate from the identification of $\sl(2)$-type Hitchin fibres for $\SO(2n+1,\C)$ with fibres of the $\pi_n^*K_X$-twisted $\SO(3,\C)$-Hitchin system on $\Sigma/\sigma$ in Theorem \ref{theo:so:biholo}.
\end{proof}

\begin{rema} It follows from Theorem \ref{theo:fibrebdle_trivial} and the last assertion in Theorem \ref{theo:SO(2n+1)_strat}, that all these fibre bundles are smoothly trivial. 
\end{rema}

\begin{coro} Let $\underline{a} \in B_{2n}(X,K_X)$ be of $\sl(2)$-type with irreducible and reduced $\Sp(2n,\C)$-spectral curve. Then $\H_{\SO(2n+1,\C)}^{-1}(\underline{a})$ has two connected components. If $a_{2n}\in H^0(X,K_X^{2n})$ has at least one zero of odd order, these two connected components are irreducible. If all zeros of $a_{2n}$ have even order, then each connected component has two irreducible components.
\end{coro}
\begin{proof}
For $\PGL(2,\C)$, the Hitchin fibres in
\[ \left( \M_{\SL(2,\C)}^{\O_X}(X,\pi_{n}^*K_X) \sqcup \M_{\SL(2,\C)}^{N}(X,\pi_{n}^*K_X) \right) / \Jac(X)[2] \] have two connected components by \cite{Ho1} Corollary 8.6 and Theorem 8.8. These results also prove that the components are irreducible in the first case. When all zeros of $a_{2n}$ have even order, each connected component has two irreducible components stemming form the two connected components of $\Prym(\tilde{\pi}_2)$. In the difference to the $\SL(2,\C)$-case, the pullback of Higgs bundles along $\tilde{\pi}_2$ is injective for $\PGL(2,\C)$ (cf. Proposition 3.12 \cite{Ho1}). Now, the general result follows from Theorem \ref{theo:so:biholo}. 
\end{proof}

In particular, Corollary \ref{coro:duality_rank1} generalizes verbatim to higher rank:
\begin{coro}\label{coro:duality} Let $\underline{a} \in B_{2n}(X,K_X)$ be of $\sl(2)$-type with  the spectral curve is irreducible and reduced. The Hitchin fibres $\H^{-1}_{\SO(2n+1,\C)}(\underline{a})$ and $\H^{-1}_{\Sp(2n,\C)}(\underline{a})$ are related as follows:
\begin{itemize} 
\item[i)]  The abelian part of the spectral data is a disjoint union of torsors over dual abelian varieties.
\item[ii)] The complex spaces of Hecke parameters are isomorphic.
\end{itemize}
\end{coro}

\section[Decoupled Hitchin equation]{Solution to the decoupled Hitchin equation through semi-abelian spectral data}\label{sec:lim} 
In this last section, we will show how to use semi-abelian spectral data for $\sl(2)$-type Hitchin fibres to produce solutions to the decoupled Hitchin equation. In a series of works of Fredrickson, Mazzeo, Swoboda, Weiss and Witt \cite{MSWW14,MSWW16,Fr18b} and independently Mochizuki \cite{Mo16}, such singular Hermitian metrics were established as limits of sequences of actual solutions to the Hitchin equation under scaling the Higgs field to infinity. We conjecture this to be true for the solutions to the decoupled Hitchin equation that we will construct. In the $\SL(2,\C)$-case, this is a theorem by \cite{Mo16}.

Let $(P,\Psi) \in \M_{G}(X,M)$. A reduction of structure group $h: X \rightarrow P/K_G$ to a maximal compact subgroup $K_G \subset G$ is called solution to the decoupled Hitchin equation, if  the Chern connection is flat and the Higgs field $\Psi$ is normal. By definition, the Higgs field $\Psi$ is normal, if 
\[ 0=[ \Psi \otimes \tau^h(\Psi) ] \in H^0\left(X,(P \times_{\Ad} \g) \otimes M^2\right),
\] where $\tau^h$ denotes the induced Cartan involution on $P \times_{\Ad} \g$. For $M=K_X$ this is equivalent to
\[ F_h=0, \quad 0 =[\Psi \wedge \tau^h(\Psi)] \in H^{(1,1)}(X,P \times_{\Ad} \g).
\] In most cases, there is no smooth solution to this equation. For $\SL(2,\C)$ it is easy to check by a local computations similar to \cite{MSWW14} Section 3.2, that $h$ is singular at all zeros of the determinant of odd order (cf. Corollary \ref{lim:rema:even_zeros_lowest_stratu}). Global solutions to the decoupled Hitchin equation can be constructed through the pushforward of a Hermitian-Einstein metric on the eigen line bundle $L \in \Prym_{\pi_n^*K_X^{-1}(D)}(\tilde{\pi})$. This method was applied for regular Hitchin fibers in \cite{MSWW14,Fr18b}.

\begin{rema} Usually the solutions of the decoupled Hitchin equation are not unique. They can be modified by applying a Hecke modification - a meromorphic gauge - at a singularity of the Hermitian metric. However, in the cases we consider, there are natural choices indicated by the construction and the known approximation results of \cite{Mo16} and \cite{Fr18b}.  
\end{rema}

\subsection{\texorpdfstring{$\Sp(2n,\C)$}{Sp(2n,C)}}

 
Before going to higher rank, we reproduce a result of Mochizuki for the $\SL(2,\C)$-case (\cite{Mo16} Section 4.3.2.) using the description of singular Hitchin fibers by abelian parameters and Hecke parameters developed in \cite{Ho1}. The resulting Hermitian metrics agree by the uniqueness statement \cite{Mo16} Lemma 4.8.

\begin{theo}[cf. \cite{Mo16} Section 4.3.2]\label{lim:theo:sol_dec_hit}
 Let $(E,\Phi) \in \M_{\SL(2,\C)}(X,M)$ with irreducible and reduced spectral curve. Let $a_2=\det(\Phi)$, $D$ its Higgs divisor and for $x \in Z(a_2)$ let $n_x:= \ord_x{a_2}-2D_x \in \N_0$. Then there exists a Hermitian metric $h_{dc}=h_{dc}(E,\Phi)$ on $E\rest_{X\setminus Z(a_2)}$ solving the decoupled Hitchin equation and inducing a non-singular Hermitian metric on $\det(E)$.
For all $x \in Z(a_2)$ there exists a coordinate $(U,z)$ centred at $x$ and a local frame of $E\rest_U$, such that the Higgs field is given by 
\[ \Phi= \begin{pmatrix} 0 & z^{D_x} \\ z^{\ord_x{a_2}-D_x} & 0 \end{pmatrix} \d z
\] and the Hermitian metric for $\ord_x(a_2) \equiv 1 \mod 2$ is given by
\[ h_{dc}=  \begin{pmatrix} g_1 \vert z \vert^{\frac{n_x}2} & g_2 z^{\frac{1-n_x}{2}} \vert z \vert^{\frac{n_x}2} \\ \bar{g_2}  \bar{z}^{\frac{1-n_x}{2}} \vert z \vert^{\frac{n_x}2} & g_1 \vert z \vert^{-\frac{n_x}2} \end{pmatrix},
\] with $g_1$ a real positive smooth function and $g_2$ a complex smooth function, such that $g_1^2-\vert g_2 \vert ^2 \vert z \vert =1$. For $\ord_x(a_2) \equiv 0 \mod 2$ with respect to such frame, the Hermitian metric is given by
\[ h_{dc}=  \begin{pmatrix} g_1 \vert z \vert^{\frac{n_x}2} & g_2 z^{\frac{-n_x}{2}} \vert z \vert^{\frac{n_x}2} \\ \bar{g_2}  \bar{z}^{\frac{-n_x}{2}} \vert z \vert^{\frac{n_x}2} & g_1 \vert z \vert^{-\frac{n_x}2} \end{pmatrix},
\] with $g_1,g_2$ real positive smooth functions, such that $g_1^2- g_2^2 =1$. In both cases, the smooth functions $g_1,g_2 \in \A_U$ are determined through the $u$-coordinate of $(E,\Phi)$ at $x$.

\end{theo}

\begin{proof}
Let $(E,\Phi) \in \H^{-1}_{\SL(2,\C)}(a_2)$ with Higgs divisor $D$. Recall the description of $(E,\Phi)$ via the semi-abelian spectral data developed in \cite{Ho1} Section 5. Let $\lambda \in H^0( \tilde{\Sigma}, \tpst M)$ the section solving the spectral equation and $\Lambda=\div(\lambda)$ its divisor. The abelian part of the spectral data is a line bundle $L \in \Prym_{M^{-1}(D)}(\tilde{\pi})$ defined by 
\[ \O(L)=\ker(\tpst \Phi -\lambda \id_{\O(\tpst E)}).
\] We recover $\tpst E$ as a Hecke transformation of 
\[ (E_L,\Phi_L):=(L \oplus \sigma^*L, \diag(\lambda,-\lambda)).
\] Up to choices of frames of $L$ at $Z(\tpst a_2)$ these Hecke transformations are para\-met\-rized by a so-called $u$-coordinate. These $u$-coordinates are the non-abelian parameters of the spectral data. 

For constructing the solution to the decoupled Hitchin equation let us fix an auxiliary parabolic structure on $L$ by introducing weights $\alpha_p:= \frac12 (\Lambda-\tpst D)_p$ for all $p \in Z(\tpst a_2)$. Then the parabolic degree $\mathrm{pdeg}(L,\alpha)=0$. Hence, there exists a Hermitian metric $h_L$ adapted to the parabolic structure that satisfies the Hermitian-Einstein equation 
\[ F_{h_L}=0
\] unique up to rescaling by a constant (see \cite{Bi96,Si90}). This induces a flat Hermitian metric $h_L+ \sigma^*h_L$ on $E_L$, such that the Higgs field $\Phi_L$ is normal. Applying the Hecke transformation to $(E_L,\Phi_L,h_L+ \sigma^*h_L)$ we obtain a Hermitian metric on $\tpst E \rest_ {X \setminus Z(a_2)}$ solving the decoupled Hitchin equation. This descends to the desired metric $h_{dc}$. 

To show that it induces a non-degenerate Hermitian metric on $\det(E)=O_X$, we compute its local shape at $Z(a_2)$. Let $x \in Z(a_2)$ be a zero of odd order and $p \in \tilde{\Sigma}$ its pre-image. By \cite{Fr18b} Proposition 3.5, we can choose a frame $s$ of $L$ around $p$, such that $h_L=\vert z \vert^{2\alpha_p}$. Such frame is unique up to multiplying with $c \in \mathsf{U}(1)$ and therefore defines a unique $u$-coordinate for $(E,\Phi)$ at $p$ (see \cite{Ho1} Proposition 5.8). We want to change the frame of $L$, such that the Higgs bundle $(E,\Phi)$ corresponds to $u=0$ respective the new frame. This guarantees the desired local shape of $\Phi$. The transformation rule for $u$-coordinates was given in \cite{Ho1} Proposition 5.8. Choosing the frame $s'= \sqrt{\frac{1+u}{1-u}} s$ the $u$-coordinate for $(E,\Phi)$ is $u'=0$. The Hermitian metric $h_L$ with respect to the frame $s'$ is given by 
\[ h_L=f \vert z \vert ^{2\alpha_p} \quad \text{with } f=\left\vert \frac{1-u}{1+u} \right\vert.
\] Applying the Hecke transformation the induced Hermitian metric on $\tpst E$ at $p$ is given by 
\[ \begin{pmatrix} (f+ \sigma^*f) \vert z \vert ^{2\alpha_p} & (f- \sigma^*f) \left(\frac{\vert z \vert}{z}\right)^{2\alpha_p} \\ (f- \sigma^*f) \left(\frac{\vert z \vert}{\bar{z}}\right)^{2\alpha_p} & (f+ \sigma^*f)\vert z \vert ^{-2\alpha_p}
\end{pmatrix}.
\] There exists $g_1,g_2 \in \A_U$, such that 
\[ \tpst g_1= f+ \sigma^*f \quad \text{and} \quad \tpst g_2= z^{-1}(f- \sigma^*f)
\] Hence, we obtain the desired local form of $h_{dc}$ at $x \in Z(a_2)$.

Using the description of the Hecke parameters at even zeros in terms of $u$-coordinates \cite{Ho1} Proposition 8.1, one can adapt this argument to the zeros of $a_2$ of even order. 

In the case of irreducible, locally reducible spectral curve $\Sigma$, the proof works in the same way as
\[ \Prym_{ IM^{-1}(D)}(\tilde{\pi}_2) \subset \{ L \in \Pic(\tilde{\Sigma} \mid L \otimes \sigma^*L= M^{-1}(D) \},
\] where $I$ is the unique non-trivial line bundle on $X$, such that $\tpst I= \O_{\tilde{\Sigma}}$.  
\end{proof}

\begin{rema} For the regular fibres of $\M_{\SL(2,\C)}(X,K_X)$, this resembles the construction of limiting metrics in \cite{Fr18b}. In difference to Fredrickson, we work with positive weights instead of negatives. This is due to the fact that Fredrickson's construction uses the line bundle $L'$ with the property $\pi_*L'=E$ to reconstruct the Higgs bundle. In terms of $L\in \Prym_{\pi_n^*K_X^{-1}}(\tilde{\pi}_2)$ it is given by $L'=:L \otimes \pi^*K_X$. To every solution $h_L$ of the Hermitian-Einstein equation on $(L,\alpha)$, as defined in the previous theorem, one obtains a solution of the Hermitian-Einstein equation on $(L',-\alpha)$ in a canonical way by $h':=h_L \vert \lambda \vert ^{-2}$. 
\end{rema}

\begin{rema}\label{rem:fixdetcase} Similar to \cite{Fr18b} Proposition 3.3, one can obtain solutions to the decoupled Hitchin equation in the fixed determinant case. One first fixes a Hermitian-Einstein metric $h_{\det(E)}$ on $\det(E)$. In this case, the eigen line bundle $L$ will be an element of 
\[ \{ L \in \Pic( \tilde{\Sigma}) \mid L \otimes \sigma^*L = \tilde{\pi}^*\det(E)M^{-1}(D) \}
\] and one can choose the Hermitian-Einstein metric $h_L$, such that 
\[ h_L \otimes \sigma^* h_L = \tpst h_{\det(E)} \left\vert \frac{\tpst a_2}{\tpst s_{D}^2} \right\vert,
\] where $s_D \in H^0(X,\O_X(D))$ is a canonical section. Then the induced solution to the decoupled Hitchin equation satisfies $\det(h_{dc})=h_{\det}$. 
\end{rema}

\begin{coro}\label{lim:rema:even_zeros_lowest_stratu} Let $(E,\Phi) \in \M_{\SL(2,\C)}(X,K_X)$, such that $\det(\Phi) \in H^0(X,K_X^2)$ has no global square root and $\Phi$ is everywhere locally diagonalizable. Then the Hermitian metric $h_{dc}$ defined in Theorem \ref{lim:theo:sol_dec_hit} is a smooth solution to the Hitchin equation on $(E,\Phi)$. 
\end{coro}

\begin{rema}\label{rema:locally_diag} Let $(E,\Phi)$ as in the previous corollary. $(E,\Phi)$ is stable by the irreducibility of the spectral curve. Hence, the rescaled Hitchin equation 
\[ F_h+t^2[\Phi\wedge\Phi^*h]=0, \quad t \in \C^*,
\] decouples and the solutions is independently of $t$ given by the Hermitian metric $h_{dc}$. Hence, this Hermitian metric is the limit of a constant sequence of solutions to the Hitchin equation along a ray to the ends of the moduli space. 
\end{rema}

\begin{proof}[Proof of Corollary \ref{lim:rema:even_zeros_lowest_stratu}]
We have $\tpst D_p= \frac12 \ord_p(\tpst a_2)= \ord_p(\lambda)$ for all $p \in Z(\tpst a_2)$. Hence all weights $\alpha_p$ of the auxiliary parabolic structure are zero. In particular, the Hermitian-Einstein metric $h_L$ is smooth and so is $h_L + \sigma^*h_L$. This Hermitian metric descends to $h_{dc}$. By assumption, $(E,\Phi)$ is in the lowest-dimensional stratum, i.e.\ there are no Hecke parameters for this stratum. Hence, the descend does not include a Hecke transformation and the smoothness is preserved.
\end{proof}

\begin{theo}[\cite{Mo16} Corollary 5.4]\label{theo:approx:Sl(2)} Let $(E,\Phi) \in \M_{\SL(2,\C)}(X,K_X)$ with irreducible and reduced spectral curve, then the solution to the decoupled Hitchin equation $h_{dc}$ is a limiting metric. In explicit, let $h_t$ be the solution to the rescaled Hitchin equation
\[ F_{h_t}+t^2[\Phi\wedge\Phi^{*{h_t}}]=0, \quad t \in \R_+,
\] then $h_t$ converges to $h_\infty$ in $C^\infty$ on any compact subset of $X\setminus Z(\det(\Phi))$ for $t \rightarrow \infty$. 
\end{theo}
\begin{proof}
For $\SL(2,\C)$-Hitchin fibres with irreducible and reduced spectral curve the auxiliary parabolic structure is uniquely determined by the condition that the singular Hermitian metric $h_\infty$ induces a non-singular Hermitian metric on $\det(E)$ (see \cite{Mo16} Lemma 4.8). Hence, $h_\infty$ coincides with the limiting Hermitian metric constructed by Mochizuki and the approximation result follows from \cite{Mo16} Corollary 5.4.
\end{proof}

Applying the biholomorphism of Theorem \ref{Theo:isom:Hitchinfibres}, we can use Theorem \ref{lim:theo:sol_dec_hit} to construct solutions to the decoupled Hitchin equation for $\Sp(2n,\C)$-Higgs bundles of $\sl(2)$-type. 

\begin{theo}\label{theo:limit:sl2type}
Let $(E,\Phi,\omega) \in \M_{\Sp(2n,\C)}(X,K_X)$ with irreducible and reduced spectral curve of $\sl(2)$-type. Let $(E_2,\Phi_2)\in \M_{\SL(2,\C)}(\Sigma/\sigma,\pi^*_n K_X)$ the corresponding $\SL(2,\C)$-Higgs bundle under the biholomorphism of Theorem \ref{Theo:isom:Hitchinfibres}. The solution to the decoupled Hitchin equation on $(E_2,\Phi_2) \in \M_{\SL(2,\C)}(\Sigma/\sigma,\pi^*_n K_X)$ induces a Hermitian metric $h_{dc}$ on $(E,\Phi,\omega)\rest_{X\setminus Z(\disc(E,\Phi,\omega))}$ solving the decoupled Hitchin equation. 
\end{theo}
\begin{proof}
Let $h_2$ be the solution to the decoupled Hitchin equation on $(E_2,\Phi_2)$ defined in Theorem \ref{lim:theo:sol_dec_hit}. Then $h':=h_2 \vert \del \pi_n \vert^{-1}$ defines a Hermitian metric on $E_2 \otimes \pi_n^*K_X^{n-1}$ singular on $Z(\det(\Phi_2)) \cup \supp R$, where $R=\div(\pi_n)$ is the ramification divisor. Recall from Theorem \ref{Theo:isom:Hitchinfibres}, that $\pi_{n*}(E_2 \otimes \pi_n^*K_X^{n-1})=E$. Hence, $\pi_{n*}h'$ defines a flat Hermitian metric on $E$ singular on 
\[ Z(a_{2n}) \cup \pi_n(\supp R)=Z(\disc(E,\Phi,\omega))
\] compatible with the symplectic form, such that $[ \pi_{n*}\Phi \wedge \pi_{n*}\Phi^{*_{\pi_{n*}h'}} ]=0$.

\end{proof}


The local building blocks for these solutions to the decoupled Hitchin equation at its singularities were already considered before. Non-zero eigenvalues of the Higgs field of higher multiplicity correspond to smooth ramification points of $\pi: \Sigma \rightarrow X$. Here the $\Sp(2n,\C)$-Higgs bundle locally looks like a Higgs bundle in a regular $\SL(2n,\C)$-Hitchin fibre. Hence, the local approximation problem is covered by \cite{Fr18b} Section 4.1. The singular points of $\Sigma$ lie on the zero section of $K_X$ and are locally given by an equation of the form 
\[ \lambda^2-z^k=0.
\] These are exactly the singularities of $\SL(2,\C)$-spectral curves. In this case, the local approximation result was proven in \cite{Mo16} Section 3. This leads to the following conjecture.

\begin{conj}\label{conj:limit:config} Let $(E,\Phi,\omega) \in \M_{\Sp(2n,\C)}(X,K_X)$ with irreducible spectral curve of $\sl(2)$-type. Then the solution $h_{dc}(E,\Phi,\omega)$ to the decoupled Hitchin equation is a limiting metric, i. e. let $h_t$ be the solution to the rescaled Hitchin equation
\[F_{h_t}+t^2[\Phi\wedge\Phi^{*{h_t}}]=0, \quad t \in \R_+,
\] then $h_t$ converges to $h_\infty$ in $C^\infty$ on any compact subset of $X\setminus Z(\disc_\sp(E,\Phi,\omega))$ for $t \rightarrow \infty$. 
\end{conj}

\subsection{\texorpdfstring{$\SO(2n+1,\C)$}{SO(2n+1,C)}}

\begin{theo}\label{theo:SO(3)_lim_config} Let $(E,\Phi,\omega) \in \M_{\SO(3,\C)}(X,K_X)$ and $a_2=\det(\Phi)$, such that the associated $\SL(2,\C)$-spectral curve is irreducible and reduced. Then there exists a metric on $(E,\Phi,\omega)\rest_{X\setminus Z(a_2)}$ solving the decoupled Hitchin equation.
\end{theo}
\begin{proof}
The adjoint representation $\Ad: \GL(2,\C) \rightarrow \SO(3,\C)$ induces a commutative diagram 
\[ \begin{tikzcd}
0 \ar[r] & \mathsf{U}(1) \ar[r]\ar[d]& \C^* \ar[r]\ar[d]& \R^+ \ar[r]\ar[d]& 0\\
0 \ar[r] & \mathsf{U}(2) \ar[r]\ar[d]& \GL(2,\C) \ar[r]\ar[d,"\Ad"]& \GL(2,\C)/\mathsf{U}(2) \ar[r]\ar[d]& 0 \\
0 \ar[r] & \SO(3) \ar[r]&  \SO(3,\C) \ar[r]& \SO(3,\C)/\SO(3) \ar[r]& 0
\end{tikzcd}
\]
A metric on the $\SO(3,\C)$-Higgs bundle $(E,\Phi,\omega)$ is a reduction of structure group to $\SO(3)$. Denoting by $P$ the $\SO(3,\C)$-frame bundle associated to $E$, it corresponds to a section of $P \times_{\SO(3,\C)} \SO(3,\C)/\SO(3)$. Let $(E',\Phi') \in \Hit^{-1}_{\GL(n,\C)}(0,a_2)$, such that its image under the map of Higgs bundle moduli spaces 
\[ \M_{\GL(n,\C)}(X,K_X) \rightarrow \M_{\SO(3,\C)}(X,K_X)
\] induced by the adjoint representation is $(E,\Phi,\omega)$. By the above diagram any Hermitian metric on $(E',\Phi')$ induces a metric on $(E,\Phi,\omega)$. Let $h_{\det}$ denote the Hermitian-Yang-Mills metric on $\det(E')$. By Remark \ref{rem:fixdetcase}, there exists a solution to the decoupled Hitchin equation $h_{dc}$ on $(E',\Phi')$, such that $\det(h_{dc})=h_{\det}$ unique up to scaling. This induces a solution to the decoupled Hitchin equation on $(E,\Phi)$ by the above diagram. Furthermore, if we choose another representative $(E'\otimes L,\Phi')$ with $L \in \Pic(X)$. Then $h_{dc}(E'\otimes L,\Phi')=h_{dc}(E',\Phi')h_L$, where $h_L$ is the Hermitian-Einstein metric on $L$. We see from the commutative diagram that the resulting metric on the $\SO(3,\C)$-bundle $E$ does not depend on this choice. 
\end{proof}

\begin{exam}\label{exam:lim_config_so(3)_sing}
Fix a Higgs divisor $D$ associated to $a_2 \in H^0(X,K_X^2)$. Let 
\[ (E,\Phi,\omega) \in \S_D \subset \Hit_{\SO(3,\C)}^{-1}(a_2)
\] be the Higgs bundle corresponding to the $u$-coordinate $0$ respective the frames fixed in the proof of Theorem \ref{lim:theo:sol_dec_hit}. For $x \in Z(a_2)$ there exists a coordinate $(U,z)$ centred at $x$ and a local frame of $E\rest_U$, such that the Higgs field is given by
\[ \Phi= \begin{pmatrix} 0 & i \sqrt{2} z^{D_x} & 0 \\ -i \sqrt{2}z^{\ord_x(a_2)-D_x} & 0 & -i \sqrt{2}z^{D_x} \\ 0 & i \sqrt{2}z^{\ord_x(a_2)-D_x} & 0 \end{pmatrix} \d z,
\] the orthogonal structure by
\[ \omega=\begin{pmatrix} && 1 \\ & 1 & \\ 1 && \end{pmatrix}
\] and the solution to the decoupled Hitchin equation is given by
\[ h_{dc}= \begin{pmatrix} \vert z \vert^{\ord_x(a_2)-D_x} & 0 & 0 \\ 0 & 1 & 0 \\ 0 & 0 & \vert z \vert ^{D_x-\ord_x(a_2)}
\end{pmatrix} .
\] 
\end{exam}

Applying Theorem \ref{theo:so:biholo}, we can use this result to obtain solutions to the decoupled Hitchin equation for $\SO(2n+1,\C)$-Higgs bundles of $\sl(2)$-type.

\begin{theo}\label{theo:hdc_SO(2n+1)} Let $(E,\Phi,\omega) \in \M_{\SO(2n+1,\C)}(X,K_X)$ be of $\sl(2)$-type with irreducible and reduced spectral curve. Then the pushforward along $\pi_{n*}$ defines a solution to the decoupled Hitchin equation $h_{dc}$. 
\end{theo}
\begin{proof}
This proof is similar to the proof of Theorem \ref{theo:limit:sl2type}.  Let $(E_3,\Phi_3,\omega_3)$ be the $\pi_n^*K_X$-twisted $\SO(3,\C)$-Higgs bundle on $\Sigma/\sigma$ corresponding to $(E,\Phi,\omega)$ under the isomorphism of Theorem \ref{theo:so:biholo}. Recall, that we recover $(E,\Phi,\omega)$ by a unique Hecke modification of 
\[ (\hat{E},\hat{\Phi}):=\left( \ker(\Phi) \oplus \pi_{n*} \left(\Ker (\Phi_3)^{\perp} \otimes \pi^*_n K_X^{n-1}\right), 0 \oplus \Phi \rest_{\Ker(\Phi_3)^{\perp}} \right),
\] where the perpendicular is taken with respect to the orthogonal structure $\omega_3$. In Theorem \ref{theo:SO(3)_lim_config}, we constructed a solution to the decoupled Hitchin equation $h_3$ on $(E_3,\Phi_3,\omega_3)$. $h_3$ induces a singular Hermitian metric on $\Ker (\Phi_3)^\perp$, which descends to a Hermitian metric $\pi_{n*}(h_3 \vert \partial \pi_n \vert^{-1})$ on $\pi_{n*} (\Ker (\Phi_3)^{\perp} \otimes \pi^*_n K_X^{n-1})$ singular at $Z(a_{2n}) \cup \supp B$. Here $B$ denotes the branch divisor of $\pi_n: \Sigma/\sigma \rightarrow X$. There is a canonical singular flat metric on $\Ker(\Phi)= K_X^{-n}(D)$ given by $\vert \frac{a_{2n}}{s_D} \vert$ singular at $Z(a_{2n})$. This defines a singular flat Hermitian metric 
\[ \vert\frac{a_{2n}}{s_D} \vert \oplus \pi_{n*}(h_3 \vert \partial \pi_n \vert^{-1})
\]  on $(\hat{E},\hat{\Phi})$, such that the Higgs field is normal and which is compatible with the singular orthogonal structure. The Hecke modification at $Z(a_{2n})$ desingularizes the orthogonal structure. The induced Hermitian metric on $(E,\Phi,\omega)$ solving the decoupled Hitchin equation is singular at $Z(a_{2n}) \cup \supp (B) = Z(\disc_{\sp}(a_2, \dots, a_{2n} ))$. 

\end{proof}

\begin{rema} In the previous proof, we used Lemma \ref{lemm:so:kernel_Hecketrafo} stating that we can reconstruct a $\SO(2n+1,\C)$-Higgs bundle in a unique way from 
\[ (\hat{E},\hat{\Phi})=\left( \Ker( \Phi ) \oplus \Ker(\Phi)^\perp, 0 \oplus \Phi \rest_{\Ker(\Phi)^\perp}\right)
\] through Hecke modification. For a $\SO(3,\C)$-Higgs bundle this gives another way to construct a solution to the decoupled Hitchin equation. An easy, but tedious computation using the local models of Theorem \ref{lim:theo:sol_dec_hit} shows that $h_{dc}$ as constructed in Theorem \ref{theo:SO(3)_lim_config} is equal to the solution of the decoupled Hitchin equation obtain from the singular flat Hermitian metric 
\[ \vert\frac{a_{2}}{s_D} \vert \oplus h_{dc} \rest_{\Ker(\Phi)^{\perp}}
\] on $(\hat{E},\hat{\Phi})$ in this way. This shows that the construction of solutions to the decoupled Hitchin equation in the previous proof is consistent.
\end{rema}

Similar to the symplectic case, the approximation of the local models of these solutions of the decoupled $\SO(2n+1,\C)$-Hitchin equation follows from the work of Mochizuki \cite{Mo16} and Fredrickson \cite{Fr18b}. Applying the same argument as in the proof of Theorem \ref{theo:SO(3)_lim_config} to the solutions of the rescaled Hitchin equation in Theorem \ref{theo:approx:Sl(2)}, one shows that the solutions to the decoupled $\SO(3,\C)$-Hitchin equation are limiting metrics. In particular, the local models of the solution to the decoupled Hitchin equation for $\SO(2n+1,\C)$ are approximated. At the branch points of $\pi_n: \Sigma/\sigma \rightarrow X$ the Higgs bundle looks like to copies of a $\SL(n,\C)$-Higgs bundle interchanged by $\omega$ plus the kernel. So the local models are described and approximated by the work of \cite{Fr18b}. This leads us to analogue of Conjecture \ref{conj:limit:config} for $\SO(2n+1,\C)$.

\begin{conj} The solutions to the decoupled Hitchin equation for $\SO(2n+1,\C)$ constructed in Theorem \ref{theo:hdc_SO(2n+1)} are limiting metrics. 
\end{conj} 

\begin{rema} In the light of the Langlands duality of Hitchin systems, it would be interesting to compute, whether the induced $L^2$-metrics on $\Hit_{\Sp(2n,\C)}^{-1}(B^{\reg})$ and $\Hit_{\SO(2n+1,\C)}^{-1}(B^{\reg})$ define the semi-flat hyperkähler metrics associated to the Langlands dual algebraically completely integrable systems.  We hope to come back to this question in the future.
\end{rema}

\subsection{Smooth trivialisation of semi-abelian spectral data}

In Section \ref{sec:sp:semi-abel} and \ref{sec:langlands}, we stratified the $\sl(2)$-type Hitchin fibres by fibre bundles over abelian torsors. Using the solutions to the Hermitian-Einstein equation discussed above, we can prove that all these fibre bundles are smoothly trivial.
 
\begin{proof}[Proof of Theorem \ref{theo:fibrebdle_trivial}] We only need to show the triviality in the $\SL(2,\C)$-case. Then it follows in all other cases by the identification of $\sl(2)$-type Hitchin fibres with fibres of an $\SL(2,\C)$- resp. $\PGL(2,\C)$-Hitchin map.
In the proof of Theorem \ref{lim:theo:sol_dec_hit}, we saw that a solution to the Hermitian-Einstein equation $h_L$ on the eigen line bundle $L \in \Prym_{K_X^{-1}(D)}(\tilde{\pi})$ with respect to some auxiliary parabolic structure induces local frames $s$ at $p \in \tilde{\pi}^{-1}Z(a_2)$, such that $h_L=\vert z \vert^{2\alpha_p}$. These frames are unique up to multiplying by a constant and therefore define unique $u$-coordinates at all $p \in Z(\tilde{\pi}^*a_2)$ (see \cite{Ho1} Proposition 5.8). $h_L$ depends smoothly on $L \in \Prym_{K_X^{-1}(D)}(\tilde{\pi})$ (see \cite{MSWW19} Proposition 3.3). Furthermore, the choice of $s$ depends smoothly on $h_L$ by the explicit argument in \cite{Fr18b} Proposition 3.5.  Hence, this defines a smooth trivialisation in the $\SL(2,\C)$-case and hence in all other cases. 
\end{proof}

\printbibliography

\end{document}